\documentclass{article}
\usepackage[top=4cm, bottom=4cm, left=3cm, right=3cm]{geometry}
\usepackage{graphics}
 \usepackage{graphicx}
 \usepackage{epsfig}
 \usepackage{amsmath}
\usepackage{amsfonts}
\usepackage{amssymb}
\usepackage{xcolor}
\usepackage{enumerate}
\usepackage{hyperref}
\hypersetup{colorlinks=true,linkcolor=blue}
\numberwithin{equation}{section}

\newtheorem{theorem}{Theorem}[section]

\newtheorem{prop}{Proposition}[section]
\newtheorem{leme}{Lemma}[section]
\newtheorem{dfnt}{Definition}[section]
\newtheorem{remark}{Remark}

\newtheorem{problem}{Problem}
\newtheorem{assumption}{Assumption}[section]

\newenvironment{proof}[1][Proof]{\textbf{#1.} }{\ \rule{0.5em}{0.5em}}

\def \dis {\displaystyle}

\def \R {\mathbb{R}}

\def \O {\mathcal{O}}

\def\dT{dx\, dt}

\def \hvarphi \widehat{\varphi}

\def\dq{dx\,dt}

\def \hvarphi \widehat{\varphi}

\begin{document}
  \title{Stackelberg strategy on a degenerate parabolic equation with missing data}
 \author{{{Landry Djomegne} \thanks{{\it
 				University of Dschang, BP 67 Dschang, Cameroon, West region,
 				email~: {\sf landry.djomegne\char64gmail.com} }}}
 	{{\quad Cyrille Kenne} \thanks{{\it Laboratoire LAMIA, Universit\'e des Antilles,
 				Campus Fouillole, 97159 Pointe-{\`a}-Pitre  Guadeloupe (FWI)-	
 				Laboratoire L3MA, UFR STE et IUT, Universit\'e
 				des Antilles,97275 Schoelcher, Martinique,
 				email~: {\sf kenne853\char64gmail.com} }}}
 {{\quad Romario Foko Tiomela} \thanks{{\it Department of Mathematics, Morgan State University, Baltimore, Maryland,USA,
 			email~: {\sf romario.foko\char64morgan.edu} }}}
 }

  \date{\today}
  
  \maketitle	
  	

\begin{abstract}
This paper deals with the hierarchic control of a degenerate parabolic equation with missing initial condition. We present a Stackelberg strategy combining the concept of null controllability with low-regret control. We assume that we can act on the system through a set of hierarchic controls. The main control called the leader is in charge of the null controllability while the second control named the follower solves an optimal control problem involving a missing data. The main novelty of this work is the derivation of a new Carleman inequality for a degenerate system, which is used in a standard way to show observability inequality of the adjoint degenerate systems.\\
\end{abstract}


\section{Introduction}\paragraph{}
Let $\Omega=(0,1)$ be a bounded open set of $\R$. Let also $\mathcal{O}$ and $\omega$ be two non empty open subsets of $\Omega$ such that $\omega\varsubsetneq\O$.  For $T > 0$, we set $Q =(0, T )\times \Omega$,
$\omega_T =(0, T )\times\omega $ and $\mathcal{O}_T=(0,T)\times\mathcal{O}$. Then, we are interested in a hierarchical strategy of the following linear degenerate parabolic equation:

  \begin{equation}\label{eq}
  \left\{
  \begin{array}{rllll}
  \dis y_t-\left(k(x)y_x\right)_{x}+a_0y  &=&v\chi_{\mathcal{O}}+h\chi_{\omega}& \mbox{in}& Q,\\
  \dis y(t,0)=y(t,1)&=&0& \mbox{on}& (0,T), \\
  \dis y(0,\cdot)&=&g&\mbox{in}&\Omega,
  \end{array}
  \right.
  \end{equation}
  where $g\in L^2(\Omega)$ is the unknown initial condition, the potential $a_0\in L^\infty(Q)$ is given. We denote by $y_t$ and $y_x$ the partial derivatives of $y$ with respect to $t$ and $x$ respectively.\par 
  In the system \eqref{eq}, $y=y(t,x)=y(t,x;h;v,g)$ is the state while $v=v(t,x)$ and $h=h(t,x)$ are two different control functions  applied respectively on $\O$ and $\omega$. These functions $v$ and $h$ are the follower and leader controls respectively. Here $\chi_{\mathcal{O}}$ and $\chi_{\omega}$ are respectively the characteristic function of the control set $\mathcal{O}$ and $\omega$.\par 
  
  The system \eqref{eq} can be used to model the dispersion of a gene in a given population (invasive species for instance). In this case, $x$ represents the gene type, and $y(t,x)$ denote the density of individuals at time $t$ and of gene type $x$ \cite{younes2014}. In this paper, the function $k$ is the diffusion coefficient which depends on the gene type and degenerate at the left hand side of its domain, i.e. $k(0)=0$, (e.g $k(x)=x^\alpha,\ \alpha>0$). In this case, we say that the system \eqref{eq} is a degenerate parabolic equation. Genetically speaking, such a property of degeneracy means that if each population is not of gene type, it cannot be transmitted to its offspring \cite{younes2014}.\par
  
  The model \eqref{eq} is also called a system with incomplete data because, the information on the initial data is not completely known. The missing term in the initial condition may expresses the fact that we have no information on the density of population at the beginning of the study.
  
 In this article, unlike the single objective case, we are in a situation that we have two objectives to achieve and for that the introduction of bi-objective optimization is essential. More precisely, we use the concept of Stackelberg competition introduced in 1934 \cite{Von1934Stackelberg}. It represents a strategy  game between two firms in which one of the firms (the leader) moves first and the second firm (the follower) moves according to the leader's strategy.\par 
 
 In the framework of PDEs, the hierarchic control was introduced by J-L. Lions in 1994 \cite{Lions1994Stackelberg1, Lions1994Stackelberg2} to study a bi-objective control problem for the wave and heat equation respectively. In theses works, the author acted on the different systems with two controls. The leader solving an approximate controllability problem while the follower solves an optimal control problem. In recent years, many other researchers have used hierarchic control in the sense of Lions; see for instance \cite{djomegnebackward, kere2017coupled, mercan1, mercan2, mercan3, Nakoulima2007}. In \cite{teresa2018, liliana2020, teresabis}, the authors used the hierarchic control which combine the controllability problem with robustness.\par 
 
 Recently in 2020 \cite{romario2018}, the authors combined the concepts of hierarchic control and low-regret control on a linear heat equation with a missing initial condition. The leader was in charge of a null controllability problem while the follower solved an optimal control problem in presence of the missing data. The authors in \cite{djomegne2018} extended the previous work to a population dynamics model with an unknown birth rate. In that work, the goal of the leader was to bring the population to extinction at time $T>0$ while the follower solved an optimal control problem with missing data consisting to bring the state of the system to a desired state.\par 
 
 All the above cited works share one thing in common: they deal with hierarchic strategies associated with non degenerate systems. To the best of our knowledge, the only work dealing with the hierarchical strategy associated to degenerate systems is the one in \cite{araruna2018stackelberg}, where the authors studied the Stackelberg-Nash strategy for some linear and semi-linear degenerate parabolic equations.\par
 
 In this paper, we extend the results obtained in \cite{romario2018} to the hierarchic strategy for a degenerate parabolic equation with missing data. This has not been done before and those changes introduce additional difficulties mainly when establishing a new Carleman inequality for the degenerate system \eqref{eq} in weighted Sobolev spaces.\par
 
We assume that we have a hierarchy in our wishes and we will describe the Stackelberg strategy for system \eqref{eq}. At this level, we assume that the degenerate system \eqref{eq} is well posed. We will prove it later.\par 
   Let $\O_d\subset \Omega$ be an open set representing an observation domain of the follower. We define the follower cost functional $J_1$ by
 
  \begin{equation}\label{all16}
  J_1(h;v,g)=J(h;v,g)-J(0;0,g)-\gamma\int_{\Omega}|g|^2\ dx,
  \end{equation}
  where
  \begin{equation}\label{secon2}
  J(h;v,g)=\int_{\O_d^T}|y(h;v,g)-z_d|^2\ dxdt+\mu\int_{\O_T}|v|^2\ dxdt.
  \end{equation}
  Here, $\alpha$ and $\gamma$ are two positive constants and $z_d\in L^2(\O_d^T)$ is the desired state with $\O_d^T=(0,T)\times \O_d$.\par 
  
 We also introduce the leader functional $J_2$
  \begin{equation}\label{main}
  J_2(h)=\int_{\omega_T}|h|^2\ dxdt.
  \end{equation}
  The aim is to choose the controls $v$ and $h$ in order to achieve two different objectives:
   \begin{itemize}
   	\item The main goal is to choose $h$ minimizing the functional $J_2$ given by \eqref{main} such that the following null controllability objective holds:
   	\begin{equation}\label{obj1}
   	y(T,\cdot;h;v,g)=0\ \mbox{in}\ \Omega.
   	\end{equation}
   	\item The second goal is the following: given the function $z_d$, we want to choose the control $v$  minimizing $J_1$ given by \eqref{all16}. This means that, throughout the interval $(0,T)$,
   	\begin{equation}\label{obj2}
   	\left.
   	\begin{array}{rll}
   	&&\mbox{the solution}\ y(t,x;h;v,g)\ \mbox{of}~~\eqref{eq}~~ \mbox{remains "not too far" from the desired target}\  z_d \ \mbox{in the }\\ 
   	&&\mbox{observability domain}\  \O_d\ \mbox{ even in presence of the unknown initial condition}\  g.
   	\end{array}
   	\right.
   	\end{equation}
   \end{itemize} 
  To achieve simultaneously \eqref{obj1} and \eqref{obj2}, the control process can be described in the following two problems: 
  
  \begin{problem}\label{pb1}$ $
  	
  	Let's fix a control $h\in L^2(\omega_T)$ (leader) and let $\gamma$ be a positive constant. Find the control $v^\gamma=v^\gamma(h)\in L^2(\O_T)$ solution of the following optimization problem
  	
  	\begin{equation}\label{lowregret}
  	\inf_{v\in L^2(\O_T)}\sup_{g\in L^2(\Omega)} J_1(h;v,g),
  	\end{equation}
  	where the functional $J_1(h;v,g)$ is given by \eqref{all16}.
  \end{problem}
  
  \begin{problem}\label{pb2}$ $
  	
  	Let $v^\gamma(h)$ be the control obtained in Problem \ref{pb1} and $y^\gamma=y(t,x;h;v^\gamma(h),0)$ be the associated state. We look for an optimal control $h^\gamma\in L^2(\omega_T)$ such that 
  	\begin{equation}
  	J_2(h^\gamma)=\min_{h\in L^2(\omega_T)}J_2(h),
  	\end{equation}
  	subject to the null controllability condition
  	\begin{equation}\label{y(T)}
  	y(T,\cdot;h;v^\gamma(h),0)=0\ \mbox{in}\  \Omega.
  	\end{equation}
  \end{problem}
 
 Problem \ref{pb1} (when $h\equiv 0$) is a classical low-regret control problem which looks for a control such that a given cost functional achieves its minimum in presence of a missing data. Such control was introduced by J-L. Lions in 1992 \cite{LJL} to solve problems with missing/incomplete data. Using the notions of no-regret and low-regret control, the author proved that the solution for a low-regret problem of type \eqref{lowregret} converges to the no-regret control solution of a problem of type \eqref{lowregret} with $\gamma=0$ called no-regret problem. There are many results in the literature using these concepts of optimal control. We refer for instance to \cite{nakoulima2004, mophou2020, nakoulima2003, nakoulima2000, nakoulima2002} for non degenerate parabolic equations with incomplete data, and to \cite{Jacob2010, kenne2020} for non degenerate population dynamics models with missing data. In \cite{kenne2021}, the authors study coupled degenerate population dynamics models with missing data. However, for a quasilinear degenerate elliptic equation, see \cite{velin2004}. We also refer to \cite{baleanu, mophoulow} for time fractional diffusion equation with incomplete data.\par 
 
 Problem \ref{pb2} (when $h\equiv g\equiv 0$) is a classical null controllability problem associated with degenerate systems. Actually, after solving the first problem, the second consists in solving a null controllability problem associated to a combination of degenerate systems. The degeneracy occurs at the boundary of the space domain. To solve the controllability problems for degenerate systems, new Carleman estimates were developed for degenerate case and used to show observability inequalities of the adjoint system and then obtain the null controllability of the desired system. We refer for instance to \cite{alabau2006, birba2016, cannarsa2005, cannarsa2008, cannarsa2016, martinez2016} and the references therein.\par 
  
 In application, the hierarchic control described in this paper assume that we can act on the system at two different locations and our objectives are the following: we want to lead the system to rest at time $T$ and additionally, we wish to maintain the state of the system not too far from $z_d$ in $\mathcal{O}_d$, along $(0,T)$.
 
\subsection{Main results}
 The first result concerning the low-regret problem (i.e. Problem \ref{pb1}) is as follows:
 \begin{theorem} \label{theo1} $ $
 	
 	Let $\Omega=(0,1)$ be a bounded subset of $\R$ with $\omega $, $\mathcal{O}$ and $\O_d$ three non empty subsets of $\Omega $ with $ \omega\varsubsetneq \mathcal{O} $. Let also  $h\in L^2(\omega_T)$. Then, for any $\gamma >0$, there exist $q^\gamma\in L^2((0,T);H^1_{k}(\Omega))$ and $p^\gamma\in L^2((0,T);H^1_{k}(\Omega)) $ such that the optimization problem \eqref{lowregret} has a unique solution $v^\gamma=v^\gamma(h)\in L^2(\O_T)$	which is characterized  by the following optimality system:
 \begin{equation}\label{II.27}
 v^{\gamma}=-\frac{1}{\mu}q^\gamma\ \mbox{in}\ \O_T,
 \end{equation}
 where $y^\gamma=y(t,x;h;v^\gamma(h),0)$, $S^\gamma=S(t,x;h;v^\gamma(h))$, $p^\gamma=p^\gamma(t,x;h)$ and $q^\gamma=q^\gamma(t,x;h)$ are respectively solutions of the following optimality system:
 
 \begin{equation}\label{II.28}
 \left\{
 \begin{array}{rllll}
 \dis y_t^\gamma-\left(k(x)y_x^\gamma\right)_{x}+a_0y^\gamma &=&\dis-\frac{1}{\mu}q^\gamma\chi_{\mathcal{O}}+h\chi_{\omega}& \mbox{in}& Q,\\
 \dis y^\gamma(t,0)=y^\gamma(t,1)&=&0& \mbox{on}& (0,T), \\
 \dis y^\gamma(0,\cdot)&=&0&\mbox{in}&\Omega,
 \end{array}
 \right.
 \end{equation}
 
 \begin{equation} \label{II.29}
 \left\{
 \begin{array}{rllll}
 \dis -S_t^\gamma-\left(k(x)S_x^\gamma\right)_{x}+a_0S^\gamma &=&y^\gamma\chi_{\O_d}& \mbox{in}& Q,\\
 \dis S^\gamma(t,0)=S^\gamma(t,1)&=&0& \mbox{on}& (0,T), \\
 \dis S^\gamma(T,\cdot)&=&0&\mbox{in}&\Omega,
 \end{array}
 \right.
 \end{equation}
 
 \begin{equation} \label{II.31}
 \left\{
 \begin{array}{rllll}
 \dis p_t^\gamma-\left(k(x)p_x^\gamma\right)_{x}+a_0p^\gamma &=&0& \mbox{in}& Q,\\
 \dis p^\gamma(t,0)=p^\gamma(t,1)&=&0& \mbox{on}& (0,T), \\
 \dis p^\gamma(0,\cdot)&=&\dis\frac{1}{\sqrt{\gamma}}S^\gamma(0,\cdot)&\mbox{in}&\Omega
 \end{array}
 \right.
 \end{equation}
 and
 \begin{equation} \label{II.30}
 \left\{
 \begin{array}{rllll}
 \dis -q_t^\gamma-\left(k(x)q_x^\gamma\right)_{x}+a_0q^\gamma &=&\dis \left(y^\gamma-z_d+\frac{1}{\sqrt{\gamma}}p^\gamma\right)\chi_{\O_d}& \mbox{in}& Q,\\
 \dis q^\gamma(t,0)=q^\gamma(t,1)&=&0& \mbox{on}& (0,T), \\
 \dis q^\gamma(T,\cdot)&=&0&\mbox{in}&\Omega.
 \end{array}
 \right.
 \end{equation}
 	
 	Moreover there exists a constant $C=C(\mu)>0$ independent of $\gamma$ such that
 	\begin{equation}\label{vgamma}
 	\|v^\gamma\|_{L^2(\O_T)}\leq C\left(\|z_d\|_{L^2(\O_d^T)}+\|h\|_{L^2(\omega_T)}\right).
 	\end{equation}
 \end{theorem}
 
 The second result, on the null controllability problem  (i.e. Problem \ref{pb2}) is stated as follows:
 
 \begin{theorem}\label{theo2} $ $
 	
 	Assume that the assumptions of Theorem \ref{theo1} hold, and $\O_d$, $\omega$ are such that $\O_d\cap \omega\neq \emptyset$. Then there exists a positive real weight function $\kappa$ (the definition of $\kappa$ will be given later) such that, for any function $z_d\in L^2(\O_d^T)$  with $\dis \frac{1}{\kappa} z_d\in L^2(\O_d^T),$  there exists a unique control $\hat{h}^\gamma\in L^2(\omega_T)$ such that if $(\hat{v}^\gamma=v^\gamma(\hat{h}^\gamma),\ \hat{y}^\gamma=y^\gamma(t,x;\hat{h}^\gamma;v^\gamma(\hat{h}^\gamma),0),\ \hat{S}^\gamma=S^\gamma(t,x;\hat{h}^\gamma;v^\gamma(\hat{h}^\gamma)),\ \hat{p}^\gamma=p^\gamma(t,x;\hat{h}^\gamma),\ \hat{q}^\gamma=q^\gamma(t,x;\hat{h}^\gamma))$ satisfies \eqref{II.27}-\eqref{II.30}, then $\dis y(T,\cdot;\hat{h}^\gamma;v^\gamma(\hat{h}^\gamma),0)=0\ \mbox{in}\  \Omega.$ Moreover,
 	
 	\begin{equation}\label{pas45}
 	\hat{h}^\gamma= \hat{\rho}^\gamma\ \mbox{in}\  \omega_T,
 	\end{equation}
 	where $\hat{\rho}^\gamma$, $\hat{\psi}^\gamma$, $\hat{\phi}^\gamma$ and $\hat{\zeta}^\gamma$ are solutions of
 	
 	\begin{equation}\label{rhogamma}
 	\left\{
 	\begin{array}{rllll}
 	\dis -\hat{\rho}_t^\gamma-\left(k(x)\hat{\rho}^\gamma_x\right)_{x}+a_0\hat{\rho}^\gamma &=&(\hat{\psi}^\gamma+\hat{\phi}^\gamma)\chi_{\O_d}& \mbox{in}& Q,\\
 	\dis \hat{\rho}^\gamma(t,0)=\hat{\rho}^\gamma(t,1)&=&0& \mbox{on}& (0,T),
 	\end{array}
 	\right.
 	\end{equation}
 	
 	\begin{equation}\label{psigamma}
 	\left\{
 	\begin{array}{rllll}
 	\dis \hat{\psi}_t^\gamma-\left(k(x)\hat{\psi}^\gamma_x\right)_{x}+a_0\hat{\psi}^\gamma &=&0& \mbox{in}& Q,\\
 	\dis \hat{\psi}^\gamma(t,0)=\hat{\psi}^\gamma(t,1)&=&0& \mbox{on}& (0,T), \\
 	\dis \hat{\psi}^\gamma(0,\cdot)&=&\dis \frac{1}{\sqrt{\gamma}}\hat{\zeta}^\gamma(0,\cdot)&\mbox{in}&\Omega,
 	\end{array}
 	\right.
 	\end{equation}
 	
 	\begin{equation} \label{phigamma}
 	\left\{
 	\begin{array}{rllll}
 	\dis \hat{\phi}_t^\gamma-\left(k(x)\hat{\phi}^\gamma_x\right)_{x}+a_0\hat{\phi}^\gamma &=&\dis -\frac{1}{\mu}\hat{\rho}^\gamma\chi_{\O}& \mbox{in}& Q,\\
 	\dis \hat{\phi}^\gamma(t,0)=\hat{\phi}^\gamma(t,1)&=&0& \mbox{on}& (0,T), \\
 	\dis \hat{\phi}^\gamma(0,\cdot)&=&0&\mbox{in}&\Omega
 	\end{array}
 	\right.
 	\end{equation}
 	and
 	\begin{equation} \label{zetagamma}
 	\left\{
 	\begin{array}{rllll}
 	\dis -\hat{\zeta}_t^\gamma-\left(k(x)\hat{\zeta}^\gamma_x\right)_{x}+a_0\hat{\zeta}^\gamma &=&\dis \frac{1}{\sqrt{\gamma}}\hat{\phi}^\gamma& \mbox{in}& Q,\\
 	\dis \hat{\zeta}^\gamma(t,0)=\hat{\zeta}^\gamma(t,1)&=&0& \mbox{on}& (0,T), \\
 	\dis \hat{\zeta}^\gamma(T,\cdot)&=&0&\mbox{in}&\Omega.
 	\end{array}
 	\right.
 	\end{equation}	
 	Moreover, there exists a constant $C=C(T,\|a_0\|_{L^\infty(Q)})>0$ independent of $\gamma$ such that
 	\begin{equation}\label{kgamma}
 	\|\hat{h}^\gamma\|_{L^2(\omega_T)}\leq C\left\|\frac{1}{\kappa} z_d\right\|_{L^2(\O_d^T)}.
 	\end{equation}
 \end{theorem}
 
\begin{remark}$ $
	
	 Any control $v^\gamma(h)$ satisfying \eqref{lowregret} is called a low-regret control for $J_1$.
	
\end{remark}

 The rest of this paper is organized as follows. In Section \ref{well}, we state some well-posedness results for the system \eqref{eq}. In Section \ref{low}, we study Problem \ref{pb1} corresponding to the low-regret control. In fact, we prove that there exists an optimal control that can be chosen for any fixed leader (control) and we also provide the optimality system that characterizes the latter optimal control. We establish in Section \ref{Carleman} the observability inequality derived from a suitable Carleman inequality associated to the adjoint state of the optimality system obtained in Problem \ref{pb1}. In Section \ref{null}, once the follower strategy has been fixed, we study Problem \ref{pb2} corresponding to the null controllability. Finally, some concluding remarks are given in Section \ref{conclusion}.

 \section{Well-posedness result}\label{well}
 
 In the sequel, the usual norm in $L^\infty(Q)$ will be denoted by $\|\cdot\|_{\infty}$.
 We make the following assumptions on the diffusion coefficient $k$:
  \begin{equation}\label{k}
  \left\{\begin{array}{llll}
  \dis k\in \mathcal{C}([0,1])\cap\mathcal{C}^1((0,1]),\ \ k>0\ \mbox{in}\ (0,1] \ \mbox{and}\ k(0)=0,\\
  \dis \exists \tau\in [0,1)\ :\ xk^\prime(x)\leq \tau k(x),\ x\in [0,1].
  \end{array}
  \right.
  \end{equation}
  Note that the above assumptions on $k$ hold if we choose $k(x)=x^\alpha$ with $0\leq\alpha<1$. Then, in this case, the system \eqref{eq} will be called a weakly degenerate system. If $1\leq \alpha < 2$, a similar study can be done provided that the Neumann condition $\left(k(x)y_x\right)(0)=0$ is considered instead, and \eqref{eq} will be called a strongly degenerate system. We refer to \cite{alabau2006} for those different definitions.

 In order to study the well-posedness of system \eqref{eq}, we introduce as in \cite{cannarsa2005, cannarsa2008, cannarsa2016} the following weighted spaces $H^1_{k}(\Omega)$ and $H^2_{k}(\Omega)$ (in the sequel, "abs. cont." stands for "absolutely continuous"):
 \begin{equation}\label{}
 \left\{\begin{array}{llll}
 \dis H^1_{k}(\Omega)=\{y\in L^2(\Omega): y\ \mbox{is abs. cont. in}\ [0,1],~ \sqrt{k}y_x\in L^2(\Omega) \text{ and } \ y(0)=y(1)=0\}\\
 \dis H^2_{k}(\Omega)=\{y\in H^1_{k}(\Omega): k(x)y_x\in H^1(\Omega)\},
 \end{array}
 \right.
 \end{equation}
 endowed respectively with the norms:
 \begin{equation}\label{}
 \left\{\begin{array}{llll}
 \dis \|y\|^2_{H^1_{k}(\Omega)}=\|y\|^2_{L^2(\Omega)}+\|\sqrt{k}y_x\|^2_{L^2(\Omega)},\ \ \ y\in H^1_{k}(\Omega),\\
 \dis \|y\|^2_{H^2_{k}(\Omega)}=\|y\|^2_{H^1_{k}(\Omega)}+\|(k(x)y_x)_x\|^2_{L^2(\Omega)},\ \ \ y\in H^2_{k}(\Omega).
 \end{array}
 \right.
 \end{equation} 
The following assumption will help us to prove the existence result of system \eqref{eq}.
\begin{assumption}$ $\label{aspt}
	
	There exists a constant $\dis \alpha>0$ such that
	$$
	a_0(t,x)\geq \alpha\ \mbox{for all}\ (t,x)\in Q.
	$$
\end{assumption}
For readers' convenience, we set
$$
\mathcal{H}:=L^2((0,T);H^1_{k}(\Omega))\cap\mathcal{C}([0,T];L^2(\Omega)).
$$
We denote by $(H^{1}_k(\Omega))'$ the topological dual space of $H^1_k(\Omega)$.
If we set
\begin{equation}\label{defWTA}
	W_k(0,T)= \left\{\rho ~:~ \rho \in L^2((0,T);H^1_k(\Omega)) ~\text{and}~ \rho_t\in L^2\left((0,T);(H^{1}_k(\Omega))^\prime\right)\right\},
\end{equation}
then $W_k(0,T)$ endowed with the norm
\begin{equation}\label{}
	\|\rho\|^2_{	W_k(0,T)}=\|\rho\|^2_{L^2((0,T);H^1_k(\Omega))}+\|\rho_t\|^2_{L^2\left((0,T);(H^{1}_k(\Omega))^\prime\right)}
\end{equation}
is a Hilbert space. Moreover, we have the continuous embedding
\begin{equation}\label{contWTA}
	W_k(0,T)\subset C([0,T],L^2(\Omega)). 
\end{equation}

%
Now, we recall the following existence result given in \cite[Page 37]{lions2013}.
\begin{theorem}\label{Theolions61} 
	Let $\left(F, \|\cdot\|_F\right)$ be a Hilbert space. Let $\Phi$ be a subspace of $F$ endowed with a pre-Hilbert scalar product $(((\cdot,\cdot)))$ and the corresponding norm $|||\cdot|||$ .  Moreover, let $E:F\times \Phi\to \mathbb{C}$ be a sesquilinear form.  Assume that the following hypothesis hold:
	\begin{enumerate}
		\item  The embedding $\Phi \hookrightarrow F$ is continuous; that is, there is a constant $C_1>0$ such that
		\begin{equation}\label{theolions1}
			\|\varphi\|_{F}\leq C_1|||\varphi|||~~\forall\; \varphi \in \Phi.
		\end{equation}
		
		\item For all $\varphi\in \Phi$, the mapping $u\mapsto E(u,\varphi)$ is continuous on $F$.
		
		\item There is a constant $C_2>0$  such that
		\begin{equation}\label{theolions2}
			{E(\varphi,\varphi)}\geq C_2 |||\varphi|||^2~~~\text{for all}~~\varphi\in \Phi.
		\end{equation}
	\end{enumerate}
	If $\varphi\mapsto L(\varphi)$ is a semi linear continuous form  on $\Phi$, then there exists a function $u\in F$ satisfying
	$$
	E(u,\varphi)=L(\varphi)~~ \text{for all} ~~\varphi\in \Phi.
	$$
\end{theorem}

The weak solution of system \eqref{eq} is defined as follows.

\begin{dfnt}\label{weaksolution}
	We shall say that a function $y\in \mathcal{H}$ is a weak solution to \eqref{eq} if the following equality holds:
	\begin{equation}\label{Defweaksolution}
		\begin{array}{lll}
			\dis -\int_Q y\phi_t \, \dq + \int_Q k(x)y_x\phi_x\,\dq+\int_{Q}a_0 y\phi\, \dq
			=\dis \int_{Q} (h\chi_{\omega}+v\chi_{\O})\, \phi\; \dq
			+\int_{\Omega} g\,\phi(0,x)\ dx,
		\end{array}
	\end{equation}
 for every 
	\begin{equation}\label{HQ}
		\phi \in \mathbb{V}=\left\{\phi\in \mathcal{H}:\, \phi_t\in L^2(Q), \, \phi(T,\cdot)=0 \hbox{ in } \Omega \right\}.
	\end{equation}
\end{dfnt}

\begin{remark}\label{rmktrace}
	We observe the following:
	\begin{enumerate}[(a)]
		\item The space $\mathbb{V}$ endowed  with the norm
		\begin{align*}
			\|\phi\|^2_{\mathbb{V}}:=\|\phi\|^2_{L^2((0,T);H^1_k(\Omega))}+\|\phi(0,\cdot)\|^2_{L^2(\Omega)}
		\end{align*}
		is a Hilbert space.
		\item If $ \varphi \in \mathbb{V}$, then $\phi_t \in L^2(Q)\hookrightarrow L^2((0,T);(H^{1}_k(\Omega))^\prime)$; consequently, $\phi \in W_k(0,T)$. Therefore,  $\phi(0,\cdot)$ and $\phi(T,\cdot)$ exist and belong to $L^2(\Omega)$.
	\end{enumerate}
\end{remark}
Using Theorem \ref{Theolions61}, we prove the following result.	 

\begin{theorem}\label{exis}$  $
	
	Assume that the hypothesis \eqref{k} and Assumption \ref{aspt} are valid. For all $(v,h)\in L^2(\O_T)\times L^2(\omega_T) $ and $ g\in L^2(\Omega)$, the system \eqref{eq} admits a unique weak solution $y=y(h;v,g)=y(t,x;h;v,g)\in \mathcal{H}$ in the sense of Definition \ref{weaksolution}.
	Moreover, there exists a constant $C=C(T,\|a_0\|_{L^\infty(Q)})>0$ such that the following estimate holds:
	\begin{equation}\label{esty1y2}
	\begin{array}{llllll}
	\dis \|y(T,\cdot)\|^2_{L^2(\Omega)}+\|y\|^2_{L^2((0,T); H^1_k(\Omega))}\leq 
	C\left(\|v\|^2_{L^2(\O_T)}+\|h\|^2_{L^2(\omega_T)}+\|g\|^2_{L^2(\Omega) }\right).
	\end{array}
	\end{equation} 	
\end{theorem}

The proof of Theorem \ref{exis} can be found in the Appendix.\\

For the rest of this paper, we state the following Hardy-Poincar\'e inequality.

\begin{prop}(Hardy-Poincar\'e inequality)\cite[Proposition 2.1]{alabau2006} $ $
	
	Assume that $k:[0,1]\longrightarrow\R_+$ belong to $\mathcal{C}([0;1])$, $k(0)=0$ and $k>0$ on $(0,1]$. Furthermore, assume that here exists $\theta\in (0,1)$ such that the function $\dis x\longmapsto\frac{k(x)}{x^\theta}$ is non-increasing in a neighbourhood of $x=0$. Then, there is a constant $\overline{C}>0$ such that for any $z$, locally absolutely continuous on $(0,1]$, continuous at $0$, satisfying $z(0)=0$ and $\dis \int_0^1 k(x)|z^\prime(x)|^2\ dx<+\infty$, the following inequality holds
	
	\begin{equation}\label{hardy}
		\int_0^1 \frac{k(x)}{x^2}|z(x)|^2\ dx<\overline{C}\int_0^1 k(x)|z^\prime(x)|^2\ dx.	
	\end{equation}
	Moreover, under the same hypothesis on $z$ and the fact that the function $\dis x\longmapsto\frac{k(x)}{x^\theta}$ is non-increasing on $(0,1]$, the inequality \eqref{hardy} holds with $\dis \overline{C}=\frac{4}{(1-\theta)^2}.$	
\end{prop}

\section{Study of Problem \ref{pb1}: low-regret problem}\label{low}
In this section, we aim to prove Theorem \ref{theo1}. Before going further, we present in the following subsection some results needed to prove the existence and uniqueness of the control $v^{\gamma}$ (follower).

\subsection{Reformulation of the optimization problem \eqref{lowregret}}

Here, we firstly  show that the optimization problem \eqref{lowregret} is equivalent to a classical optimal control problem. We state and prove a result allowing us to obtain a decomposition of the functional $J$ given by \eqref{secon2}.
\begin{leme}$ $
	
	Let $(v,h)\in L^2(\O_T)\times L^2(\omega_T)$ and $g\in L^2(\Omega)$. Then, we have:
	\begin{equation}\label{decom2}
	\begin{array}{rll}
	\dis J(h;v,g)=J(0;0,g)+ J(h;v,0)-\|z_d\|^2_{L^2(\O_d^T)}+\dis 2\int_{\Omega}g\ S(0,x;h;v)\ dx,
	\end{array}
	\end{equation}
	where $S(h;v)=S(t,x;h;v)\in L^2((0,T);H^1_{k}(\Omega))$ is solution of
	\begin{equation}\label{S(v)}
	\left\{
	\begin{array}{rllll}
	\dis -S_t-\left(k(x)S_x\right)_{x}+a_0S  &=&y(h;v,0)\chi_{\O_d}& \mbox{in}& Q,\\
	\dis S(h;v)(t,0)=S(h;v)(t,1)&=&0& \mbox{on}& (0,T), \\
	\dis S(h;v)(T,\cdot)&=&0&\mbox{in}&\Omega.
	\end{array}
	\right.
	\end{equation}	
\end{leme}

\begin{proof}
Let $y=y(h;v,g)=y(t,x;h;v,g)$ be the solution of \eqref{eq}. Then we write
\begin{equation}\label{decom}
y(h;v,g)=y(h;v,0)+y(0;0,g),
\end{equation}
where $y(h;v,0)$ and $y(0;0,g)$ are respectively solutions of
\begin{equation}\label{II.2}
\left\{
\begin{array}{rllll}
\dis y_t(h;v,0)-\left(k(x)y_x(h;v,0)\right)_{x}+a_0y(h;v,0)  &=&v\chi_{\mathcal{O}}+h\chi_{\omega}& \mbox{in}& Q,\\
\dis y(t,1;h;v,0)=y(t,0;h;v,0)&=&0& \mbox{on}& (0,T), \\
\dis y(0,x;h;v,0)&=&0&\mbox{in}&\Omega
\end{array}
\right.
\end{equation}
and
\begin{equation}\label{II.3}
\left\{
\begin{array}{rllll}
\dis y_t(0;0,g)-\left(k(x)y_x(0;0,g)\right)_{x}+a_0y(0;0,g)  &=&0& \mbox{in}& Q,\\
\dis y(t,1;0;0,g)=y(t,0;0;0,g)&=&0& \mbox{on}& (0,T), \\
\dis y(0,\cdot;0;0,g)&=&g&\mbox{in}&\Omega.
\end{array}
\right.
\end{equation}
Since  $g\in L^2(\Omega)$ and $(v,h)\in L^2(\O_T)\times L^2(\omega_T)$, we know that $y(h;v,0)$ and $y(0;0,g)$  belong to $L^2((0,T);H^1_k(\Omega)).$   Using the decomposition of the state equation \eqref{decom}, we obtain

\begin{eqnarray}\label{II.5}
	J(h;v,g)=J(0;0,g)+J(h;v,0) -\|z_d\|^2_{L^2(\O_d^T)}+2\int_{\O_d^T} y(h;v,0)y(0;0,g)\ dxdt,
\end{eqnarray}
where
\begin{subequations}\label{Jk}
	\begin{alignat}{11}
	J(h;v,0)=\int_{\O_d^T}|y(h; v,0)-z_d|^2\ dxdt+\mu\int_{\O_T}|v|^2\ dxdt,\\
	J(0;0,g)=\int_{\O_d^T}|y(0; 0,g)-z_d|^2\ dxdt.
	\end{alignat}
\end{subequations}
Now, if we multiply the first equation in \eqref{S(v)} by $y(0;0,g)$ and integrate by parts over $Q$, we obtain
\begin{equation*}
\int_{\O_d^T}y(h; v,0) y(0; 0,g)\ dxdt =\dis \int_{\Omega}S(0,x;h;v) \,g\, dt.
\end{equation*}
Combining this latter equality with \eqref{II.5}, we deduce \eqref{decom2}.		
\end{proof}

Using the previous lemma, we have the following result:

\begin{leme}\label{leme2}$ $
	
	Let $h\in L^2(\omega_T)$ and $\gamma>0$. Then, the optimization problem \eqref{lowregret} is equivalent to the following optimal control problem: find $v^\gamma=v^\gamma(h)\in L^2(\O_T)$ such that
	\begin{equation}\label{II.14}
	J^{\gamma}(v^{\gamma}) =\inf_{v\in L^2(\O_T)}J^{\gamma }(v),
	\end{equation}
	 where
	 \begin{equation} \label{II.12a}
	 J^{\gamma}(v) =\dis J(h;v,0)-\|z_d\|^2_{L^2(\O_d^T)}+\frac{1}{\gamma}\left\|S\left(0,\cdot;h,v\right)\right\| _{L^{2}(\Omega)}^{2}.
	 \end{equation}
\end{leme}

\begin{proof}
Using the decomposition \eqref{decom2}, we have 	
$$
\begin{array}{lll}
\dis \sup_{g\in L^2(\Omega)} J_1(h;v,g)
&=&\dis \sup_{g\in L^2(\Omega)}\left\{ J(h;v,g)-J(0;0,g)-\gamma \left\|g\right\|_{L^{2}(\Omega)}^{2}\right\}\\
&=&\dis J(h;v,0)-\|z_d\|^2_{L^2(\O_d^T)}
\dis +2\sup_{g\in L^2(\Omega)}\left\{\int_{\Omega}S(0,x;h;v) \,g \ dx-\frac{\gamma }{2}\left\|g\right\| _{L^{2}(\Omega)}^{2}\right\}.
\end{array}
$$
By means of the Fenchel-Legendre transform, we obtain
$$
\begin{array}{lll}
\dis 2\sup_{g\in L^2(\Omega)}\left\{\int_{\Omega}S(0,x;h;v) \,g\ dx-\frac{\gamma }{2}\left\|g\right\| _{L^{2}(\Omega)}^{2}\right\}=\frac{1}{\gamma}\left\|S\left(0,\cdot;k,v\right)\right\| _{L^{2}(\Omega)}^{2}.
\end{array}
$$
Therefore,

\begin{equation} \label{II.12}
\begin{array}{lll}
\dis \sup_{g\in L^2(\Omega)}J_1(h;v,g)
\dis &=&\dis J(h;v,0)-\|z_d\|^2_{L^2(\O_d^T)}+\frac{1}{\gamma}\left\|S\left(0,\cdot;k,v\right)\right\| _{L^{2}(\Omega)}^{2}\\
&=&J^{\gamma }(v).
\end{array}
\end{equation}	
Consequently, the optimization problem \eqref{lowregret} is equivalent to the standard optimal control problem \eqref{II.14}.
\end{proof}	

\begin{remark}\label{remNoregret} If we consider the functional \eqref{all16} with $\gamma=0$, then optimization problem \eqref{lowregret} becomes
	\begin{equation}\label{noregret}
	\inf_{\atop v\in L^2(\O_T)}\sup_{g\in L^2(\Omega)}\left[J(h;v,g)-J(0;0,g)\right].
	\end{equation}
	Then we deal with the no-regret control problem. Therefore in view of \eqref{decom2}, the no-regret control denoted $\hat{v}$ belongs to the set
	\begin{equation}
	\mathcal{U}=\left\{v\in  L^2(\mathcal{O}_T) \hbox{ such that }\dis \int_{\Omega}S(0,x;h;v) \,g \,dx =0, \quad \forall g\in L^2(\Omega)\right\}.
	\label{defU}
	\end{equation}
\end{remark}

\subsection{Proof of Theorem \ref{pb1}.}	
 To prove Theorem \ref{pb1}, we proceed in three steps.\\
\textbf{Step 1.} We prove that for any $h\in L^2(\omega_T)$ and $\gamma>0$, the optimization problem \eqref{lowregret} has a unique solution $v^\gamma=v^\gamma(h)\in L^2(\O_T)$.\par 

Solving the optimization problem \eqref{lowregret} is equivalent to solve the minimization problem \eqref{II.14} (see Lemma \ref{leme2}).\\
For any $v\in L^2(\O_T)$, we have $J^{\gamma }(v)\geq -\|z_d\|_{L^2(\O_d^T)}^2$. Indeed by taking $v=-h\chi_{\omega}$ (knowing that $\omega\varsubsetneq \O$), we obtain
$$
J^\gamma(v)=\mu\|h\|^2_{L^2(\omega_T)}\geq-\|z_d\|_{L^2(\O_d^T)}^2.
$$ 
Consequently, the set $\left\{J^{\gamma }(v):\ J^{\gamma }(v)\geq -\|z_d\|_{L^2(\O_d^T)}^2,\ v\in L^2(\O_T)\right\}$ is a nonempty and lower bounded set of $\R$. Then, the minimum of $J^\gamma$, $j=\inf_{\atop v\in L^2(\omega_T)}J^\gamma(v)$ exists
and there is a minimizing sequence $v_n \in L^2(\O_d^T)$ such that 
$$
\lim\limits_{n\rightarrow\infty}J^\gamma(v_n)=j.
$$
Using classical arguments (see e.g. \cite{romario2018, djomegne2018,kenne2020}), we prove that the minimization problem \eqref{II.14} admits a unique solution. Therefore the optimization problem \eqref{lowregret} has a unique solution.\\
\textbf{Step 2.} Now, we prove that the solution $v^\gamma$ of the optimization problem \eqref{lowregret} (or equivalently \eqref{II.14}) is characterized by \eqref{II.27}-\eqref{II.30}.\par 

To characterize the optimal control $v^{\gamma }$, we write the Euler-Lagrange optimality conditions:%
\begin{equation}\label{euler}
\underset{\lambda \rightarrow 0}{\lim }\frac{J^{\gamma }( v^{\gamma}+\lambda v) -J^{\gamma }(v^{\gamma }) }{\lambda }=0,\quad \forall v\in L^2(\O_T).
\end{equation}
After some calculations,  \eqref{euler} gives,
\begin{equation}\label{II.21}
\begin{array}{lll}
0&=&\dis \int_{\O_d^T}\bar{y}\left( y\left(h; v^{\gamma },0\right) -z_{d}\right)
\dq+\mu\int_{\O_T}v^\gamma v\ \dq\\
\\
&+&\dis\frac{1}{\gamma }\int_{\Omega}\bar{S}\left(0,x;h;v\right) S\left(0,x;h;v^{\gamma}\right)\ dx, \quad \forall v\in L^{2}\left( \O_T\right),
\end{array}
\end{equation}
where $\bar{y}=\bar{y}(t,x;0;v,0)$ and $\bar{S}(0;v)=\bar{S}(t,x;0;v)$ are respectively  solutions of
\begin{equation}\label{II.19}
\left\{
\begin{array}{rllll}
\dis \bar{y}_t-\left(k(x)\bar{y}_x\right)_{x}+a_0\bar{y} &=&\dis v\chi_{\mathcal{O}}& \mbox{in}& Q,\\
\dis \bar{y}(t,0)=\bar{y}(t,1)&=&0& \mbox{on}& (0,T), \\
\dis \bar{y}(0,\cdot)&=&0&\mbox{in}&\Omega
\end{array}
\right.
\end{equation}
and
\begin{equation}  \label{II.20}
\left\{
\begin{array}{rllll}
\dis -\bar{S}_t-\left(k(x)\bar{S}_x\right)_{x}+a_0\bar{S} &=&\bar{y}\chi_{\O_d}& \mbox{in}& Q,\\
\dis \bar{S}(t,0)=\bar{S}(t,1)&=&0& \mbox{on}& (0,T), \\
\dis \bar{S}(T,\cdot)&=&0&\mbox{in}&\Omega.
\end{array}
\right.
\end{equation}
To interpret \eqref{II.21}, we use $p^\gamma$ and $q^\gamma$ respectively solutions of \eqref{II.31} and \eqref{II.30}. So if we multiply the first equation of \eqref{II.19} and \eqref{II.20} respectively by $ q^{\gamma }$ and $\dis \frac{1}{\sqrt{\gamma}}p^{\gamma }$ and integrate by parts over $Q$, we respectively obtain:
\begin{equation}
\dis \int_{\O_d^T} \bar{y}\,\left(y(h; v^{\gamma },0) -z_{d}+\frac{1}{\sqrt{\gamma }}p^{\gamma
}\right) \dq=\int_{\O_T} v\,q^{\gamma }\dq
\label{II.24}
\end{equation}
and
\begin{equation}
\frac{1}{{\gamma }}\int_{\Omega} \bar{S}(0,x;h;v) \,S(
0,x;h;v^{\gamma })\ dx =\dis \frac{1}{\sqrt{\gamma}}\int_{\O_d^T} \bar{y}\,p^{\gamma }\dq.  \label{II.25}
\end{equation}

Combining \eqref{II.21}, \eqref{II.24} and (\ref{II.25}) we obtain:

$$
\int_{\O_T}\left( \mu v^{\gamma }+q^{\gamma }\right) v\ \dq=0,\;\;\forall v\in L^{2}(\O_T).
$$

Therefore
$$
v^{\gamma }=-\frac{1 }{\mu}q^{\gamma}\text{ in } \O_T.
$$	
\textbf{Step 3.} To complete the proof of Theorem \ref{pb1}, we establish in the following Proposition, the estimate \eqref{vgamma} and the associated states.

\begin{prop}\label{pro2}
	Let $h\in L^2(\omega_T)$ be given. Let also $v^{\gamma}=v^\gamma(h) \in L^{2}(\O_T)$ be the solution of \eqref{lowregret} (or equivalently \eqref{II.14}). Let also $( v^{\gamma },\ y^{\gamma },\ S^{\gamma },\ p^{\gamma },\ q^{\gamma })$ be the unique  solution of \eqref{II.27}-\eqref{II.30}. Then, there exists a constant $C=C(T, \|a\|_{L^\infty(Q)}, \mu)>0$ independent of $\gamma$ such that
	\begin{subequations}\label{kam}
		\begin{alignat}{11}
		\|v^\gamma\|_{L^2(\O_T)}&\leq &C(\mu)(\|h\|_{L^2(\omega_T)}+\|z_d\|_{L^2(\O_d^T)}),\label{estug}\\
		\|y^\gamma\|_{L^2((0,T);H^1_k(\Omega))}&\leq &C(\|h\|_{L^2(\omega_T)}+\|z_d\|_{L^2(\O_d^T)}),\label{estyg}\\
		\|S^\gamma\|_{L^2((0,T);H^1_k(\Omega))}&\leq &C(\|h\|_{L^2(\omega_T)}+\|z_d\|_{L^2(\O_d^T)}),\label{estSg}\\
		\|p^\gamma\|_{L^2((0,T);H^1_k(\Omega))}&\leq &C(\|h\|_{L^2(\omega_T)}+\|z_d\|_{L^2(\O_d^T)}),\label{estpg}\\
		\left\|\frac{1}{\sqrt{\gamma}}p^\gamma\right\|_{L^2(Q)}&\leq &C(\|h\|_{L^2(Q_\omega)}+\|z_d\|_{L^2(\O_d^T)}),\label{estpgbis}\\
		\|q^\gamma\|_{L^2((0,T);H^1_k(\Omega))}&\leq &C(\|k\|_{L^2(\omega_T)}+\|z_d\|_{L^2(\O_d^T)}),\label{estqg}\\
		\dis \frac{1}{\sqrt{\gamma} }\left\|S(0,\cdot;v^\gamma\right\|_{L^{2}(\Omega)}&\leq &C(\mu)(\|h\|_{L^2(\omega_T)}+\|z_d\|_{L^2(\O_d^T)})\label{estS0g},\\
		\dis \left\|S(0,\cdot;v^\gamma\right\|_{L^{2}(\Omega)}&\leq &\sqrt{\gamma}C(\mu)(\|h\|_{L^2(\omega_T)}+\|z_d\|_{L^2(\O_d^T)}).\label{estS0gbis}
		\end{alignat}	
	\end{subequations}
\end{prop}

\begin{proof}
	It is clear that from \eqref{estS0g}, we have \eqref{estS0gbis}. Since $v^{\gamma}=v^{\gamma }(h) \in L^{2}(\O_T)$ is  the solution of \eqref{II.14}, we have:
	$$J^\gamma(v^\gamma)\leq J^\gamma(v),\, \forall v\in L^{2}(\O_T).$$
	Hence, we take $v=-h\chi_\omega$ and since $\omega\varsubsetneq\O$, we obtain
	$$
	J^\gamma(v^\gamma)\leq J^\gamma(-h)=\mu\|h\|^2_{L^2(\omega_T)}.
	$$
	It then follows from the definition of $J^\gamma$ given  by  \eqref{II.12}  that
	$$ 
	J(h;v^\gamma,0)+\frac{1}{\gamma }\left\| S(0,\cdot;v)\right\|_{L^{2}(\Omega)}^{2}\leq \|z_d\|^2_{L^2(\O_d^T)}+\mu\|h\|^2_{L^2(\omega_T)},
	$$
	which in view of  \eqref{Jk} implies that
	\begin{subequations}
		\begin{alignat}{11}
		\|y(h;v^\gamma,0)\|_{L^2(Q)}&\leq& \|z_d\|_{L^2(\O_d^T)}+\sqrt{\mu}\|h\|_{L^2(\omega_T)},\label{intera}\\
		\|v^\gamma\|_{L^2(\O_T)}&\leq& \frac{1}{\sqrt{\mu}}\|z_d\|_{L^2(\O_d^T)}+\|h\|_{L^2(\omega_T)},\label{interb}\\
		\frac{1}{\sqrt{\gamma} }\left\|S(0,\cdot;v^\gamma) \right\|_{L^{2}(\Omega)}&\leq& \|z_d\|_{L^2(\O_d^T)}+\sqrt{\mu}\|h\|_{L^2(\omega_T)}\label{interc}.
		\end{alignat}	
	\end{subequations}
Hence, we obtain from \eqref{interb} and \eqref{interc}, the relations \eqref{estug} and \eqref{estS0g}. In view of  \eqref{estug} and \eqref{II.28}, we deduce \eqref{estyg}. Using \eqref{II.29} and \eqref{intera}, we obtain \eqref{estSg}. From \eqref{estS0g} and \eqref{II.31}, we deduce \eqref{estpg}. 

Now, we want to establish the estimate \eqref{estpgbis} for $\dfrac{1}{\sqrt{\gamma }}%
p^{\gamma } $. Combining \eqref{II.21} and \eqref{II.25}, we have:

\begin{equation}\label{inter2}
\begin{array}{lll}
0&=&\dis \int_{\O_d^T}\bar{y}\left( y\left( h;v^{\gamma },0\right)-z_{d}\right)\dq+\mu\int_{\O_T}v^\gamma v\ \dq\\
\\
&+&\dis \frac{1}{\sqrt{\gamma}}\int_{\O_d^T} \bar{y}\,p^{\gamma } \dq,\quad \forall v\in L^2(\O_T).
\end{array}
\end{equation}
Consider the  following set
\begin{equation}\label{defEE}
\mathcal{E}=\left\{\bar{y}(v),\quad  v\in L^2(\O_T)\right\}.
\end{equation}
Then  $\mathcal{E}\subset L^2(Q)$. Define on $\mathcal{E}\times \mathcal{E}$ the inner product:

\begin{equation}\label{prodscalaireE}
\begin{array}{llll}
\dis \langle \bar{y}(v),\bar{y}(w)\rangle_{\mathcal{E}}= \dis \int_{\O_T} vw \, \dq+\dis \int_Q\bar{y}(v)\bar{y}(w)\ \dq,\ \ 
\forall\, \bar{y}(v),\bar{y}(w)\in \mathcal{E}.
\end{array}
\end{equation}
Then $\mathcal{E}$ endowed with the norm
\begin{equation}\label{normE}
\|\bar{y}(v)\|^2_{\mathcal{E}}= \|v\|_{L^2(\O_T)}^2+ \|\bar{y}(v)\|_{L^2(Q)}^2,\ \ \forall \bar{y}(v)\in \mathcal{E}
\end{equation}
is an Hilbert space.

We set $T_\gamma(v^\gamma)=\dis \frac{1}{\sqrt{\gamma}}p^{\gamma}$. Then in view of \eqref{inter2}, we have for any $ v\in L^2(\O_T),$
\begin{equation}\label{eqt81final1}
\begin{array}{lllll}
\dis \int_{\O_d^T}T_\gamma(v^\gamma)\bar{y}(v)\  \dq&=&-\dis \int_{\O_d^T}\bar{y}\left( y\left(h; v^{\gamma },0\right) -z_{d}\right)\ \dq-\mu\int_{\O_T}v^\gamma v\ \dq.
\end{array}
\end{equation}

In view of \eqref{estug} and \eqref{intera}, we have
\begin{equation}\label{eqt81final1inter1}
\begin{array}{llll}
\left|\dis-\dis \int_{\O_d^T}\bar{y}\left( y(h;v^{\gamma },0) -z_{d}\right\}\ \dq-\mu\int_{\O_T}v^\gamma v\ \dq\right|\leq C\|\bar{y}(v)\|_{\mathcal{E}},
\end{array}
\end{equation}
where $C=C\left(\|z_d\|_{L^2(\O_d^T)}, \|h\|_{L^2(\omega_T)},\mu\right)>0$ is a constant independent of $\gamma$.
It then follows from \eqref{eqt81final1} and \eqref{eqt81final1inter1} that
$$
\left|\int_{\O_d^T}T_\gamma(v^\gamma)\bar{y}(v)\  \dq\right|\leq C \|\bar{y}(v)\|_{\mathcal{E}}.
$$
This means that
\begin{equation*}\label{eqt81final1inter2}\left\|T_\gamma(v^\gamma)\right\|_{\mathcal{E}^\prime}=
\left\|\frac{1}{\sqrt{\gamma}}p^{\gamma}\right\|_{\mathcal{E}^\prime}\leq C.
\end{equation*}
In particular,
\begin{equation*}
\left\|\frac{1}{\sqrt{\gamma}}p^{\gamma}\right\|_{L^2(Q)}\leq C.
\end{equation*}
So, we get  the estimate \eqref{estpgbis}.\\
Using \eqref{estyg} and \eqref{estpgbis}, we deduce from \eqref{II.30} the estimate \eqref{estqg}.  The proof of Theorem \ref{pb1} is complete.

\end{proof}

\begin{remark}
	
Note that with the estimates \eqref{estug}-\eqref{estS0gbis} obtained in Proposition \ref{pro2}, we can extract subsequences still denoted
by $v^\gamma$, $y^\gamma$, $S^\gamma$, $p^\gamma$ and $q^\gamma$ such that when $\gamma \rightarrow 0$ we have the following convergences:
\begin{eqnarray*}
v^{\gamma}&\rightharpoonup& \hat{v}\text{ weakly in }L^{2}(\O_T), \\
y^{\gamma }&\rightharpoonup& \hat{y}\text{  weakly in }L^2((0,T);H^1_k(\Omega)), \\
S^{\gamma }&\rightharpoonup& \hat{S}\text{  weakly in }L^2((0,T);H^1_k(\Omega)),\\
q^{\gamma }&\rightharpoonup& \hat{q}\text{ weakly  in }L^2((0,T);H^1_k(\Omega)),\\
p^{\gamma }&\rightharpoonup& \hat{p}\text{  weakly in }L^2((0,T);H^1_k(\Omega)),\\
\dis \frac{1}{%
	\sqrt{\gamma} } S\left(0,x;h;v^\gamma\right) &\rightharpoonup& \varpi_1\text{  weakly in }L^{2}(\Omega),\\
S\left(0,x;h;v^\gamma\right) &\rightharpoonup& 0\text{  weakly in }L^{2}(\Omega), \label{22bis}\\
\frac{1}{\sqrt{\gamma}}p^{\gamma}&\rightharpoonup& \varpi_2\text{  weakly in }L^{2}(Q).
\text{}  \label{22ter}\\
\end{eqnarray*}	
Using the previous convergences, we can take the limit as $\gamma \rightarrow 0$ in the optimality system of Theorem \ref{pb1} and obtain that the low-regret control $v^\gamma$ converges toward the no-regret control $\hat{v}=\hat{v}(h) \in L^{2}(\O_T)$  which belongs to the set $\mathcal{U}$ (defined in Remark \ref{remNoregret}).
%
%
 However, the no-regret control $\hat{v}(h)$, the functions $\varpi_1$ and $\varpi_2$ do not depend linearly on the control $h$. This is why in Section \ref{null}, we study the null controllability of the state equation associated to the low-regret control $v^\gamma$, i.e. to the system \eqref{II.27}-\eqref{II.30}.
\end{remark}

\section{Carleman inequality}\label{Carleman}

In this section we establish an observability inequality that allows us to prove the null controllability of system \eqref{II.28}-\eqref{II.30}. We recall that the null controllability problem is related to the observability of a proper adjoint system. 
Now, for $\rho^\gamma_T\in L^2(\Omega)$, we consider the adjoint system of \eqref{II.28}-\eqref{II.30}:

\begin{equation}\label{rho}
\left\{
\begin{array}{rllll}
\dis -\rho_t^\gamma-\left(k(x)\rho^\gamma_x\right)_{x}+a_0\rho^\gamma &=&(\psi^\gamma+\phi^\gamma)\chi_{\O_d}& \mbox{in}& Q,\\
\dis \rho^\gamma(t,0)=\rho^\gamma(t,1)&=&0& \mbox{on}& (0,T), \\
\dis \rho^\gamma(T,\cdot)&=&\rho^\gamma_T&\mbox{in}&\Omega,
\end{array}
\right.
\end{equation}

\begin{equation}\label{psi}
\left\{
\begin{array}{rllll}
\dis \psi_t^\gamma-\left(k(x)\psi^\gamma_x\right)_{x}+a_0\psi^\gamma &=&0& \mbox{in}& Q,\\
\dis \psi^\gamma(t,0)=\psi^\gamma(t,1)&=&0& \mbox{on}& (0,T), \\
\dis \psi^\gamma(0,\cdot)&=&\dis \frac{1}{\sqrt{\gamma}}\zeta^\gamma(0,\cdot)&\mbox{in}&\Omega,
\end{array}
\right.
\end{equation}

\begin{equation} \label{phi}
\left\{
\begin{array}{rllll}
\dis \phi_t^\gamma-\left(k(x)\phi^\gamma_x\right)_{x}+a_0\phi^\gamma &=&\dis -\frac{1}{\mu}\rho^\gamma\chi_{\O}& \mbox{in}& Q,\\
\dis \phi^\gamma(t,0)=\phi^\gamma(t,1)&=&0& \mbox{on}& (0,T), \\
\dis \phi^\gamma(0,\cdot)&=&0&\mbox{in}&\Omega,
\end{array}
\right.
\end{equation}
and
\begin{equation} \label{zeta}
\left\{
\begin{array}{rllll}
\dis -\zeta_t^\gamma-\left(k(x)\zeta^\gamma_x\right)_{x}+a_0\zeta^\gamma &=&\dis \frac{1}{\sqrt{\gamma}}\phi^\gamma& \mbox{in}& Q,\\
\dis \zeta^\gamma(t,0)=\zeta^\gamma(t,1)&=&0& \mbox{on}& (0,T), \\
\dis \zeta^\gamma(T,\cdot)&=&0&\mbox{in}&\Omega.
\end{array}
\right.
\end{equation}
If we set $\varrho^\gamma=\phi^\gamma+\psi^\gamma$, then in view of \eqref{psi} and \eqref{phi}, $\varrho^\gamma$ is solution of
\begin{equation}\label{varrho}
	\left\{
	\begin{array}{rllll}
		\dis \varrho^\gamma_t-\left(k(x)\varrho^\gamma_x\right)_{x}+a_0\varrho^\gamma &=&\dis -\frac{1}{\mu}\rho^\gamma\chi_{\O}& \mbox{in}& Q,\\
		\dis \varrho^\gamma(t,0)=\varrho^\gamma(t,1)&=&0& \mbox{on}& (0,T), \\
		\dis \varrho^\gamma(0,\cdot)&=&\dis \frac{1}{\sqrt{\gamma}}\zeta^\gamma(0,\cdot)&\mbox{in}&\Omega,
	\end{array}
	\right.
\end{equation}
where $\rho^\gamma$ is the solution of 
\begin{equation}\label{rho1}
	\left\{
	\begin{array}{rllll}
		\dis -\rho^\gamma_t-\left(k(x)\rho^\gamma_x\right)_{x}+a_0\rho^\gamma &=&\dis \varrho^\gamma\chi_{\O_d}& \mbox{in}& Q,\\
		\dis \rho^\gamma(t,0)=\rho^\gamma(t,1)&=&0& \mbox{on}& (0,T), \\
		\dis \rho^\gamma(T,\cdot)&=&\rho_T&\mbox{in}&\Omega.
	\end{array}
	\right.
\end{equation}

\begin{remark}$ $
	
For the sake of simplicity, in this section, we will omit the gamma in system \eqref{rho}-\eqref{rho1}. Instead of $\rho^\gamma$, $\psi^\gamma$, $\phi^\gamma$ and $\zeta^\gamma$, we will be using $\rho$, $\psi$, $\phi$ and $\zeta$ throughout this section.	
\end{remark}
Classically, to establish Carleman inequality, we state first some weight functions according to the nature of the model. In our case, these functions are stated in follow:\\
since $\O_d\cap \omega\neq \emptyset$, then, there exists a non-empty open set $\omega_1\Subset \O_d\cap \omega$. Let us introduce the function $\sigma$ given by
\begin{equation}\label{sigma}
\left\{
\begin{array}{llll}
\sigma\in \mathcal{C}^2([0,1]), \sigma(x)>0\quad\text{in}\quad (0,1),\quad \sigma(0)=\sigma(1)=0,\\
\sigma_x(x)\neq 0\quad \text{in}\quad [0,1]\setminus\omega_0,
\end{array}
\right.
\end{equation}
where $\omega_0\Subset\omega_1\Subset\O_d\cap \omega$ is an open subset. We refer to \cite{FursikovImanuvilov}  for the existence of such a function $\sigma$. \\
Let  $\tau\in [0,1)$ be as in the assumption \eqref{k} and  $r, d\in \R$ be such that 
\begin{equation}\label{condrd}
r\geq \frac{4ln(2)}{\|\sigma\|_{\infty}} \hbox{ and } d\geq \frac{5}{k(1)(2-\tau)}.
\end{equation}
If $r$ and $d$ verify \eqref{condrd}, then the interval $\dis I= \left[\frac{k(1)(2-\tau)(e^{2r\|\sigma\|_{\infty}}-1)}{d\ k(1)(2-\tau)-1},
\frac{4(e^{2r\|\sigma\|_{\infty}}-e^{r\|\sigma\|_{\infty}})}{3d}\right]$
is non-empty (see  \cite{birba2016}). We can then choose $\lambda$ in this interval and for $r,\ d$ satisfying \eqref{condrd}; let's define the following functions:
\begin{equation}\label{functcarl}
\left\{
\begin{array}{llll}
\dis \Theta(t)=\frac{1}{(t(T-t))^4},\quad \forall t\in (0,T),\ \ \ \delta(x):=\dis \lambda\left(\int_{0}^{x}\frac{y}{k(y)}\ dy-d\right),\\
\\
\dis\varphi(t,x):=\Theta(t)\delta(x),\quad \eta(t,x):=\Theta(t)e^{r\sigma(x)},\\
\\
\Psi(x)=\left(e^{r\sigma(x)}-e^{2r\|\sigma\|_\infty}\right),\ \ \ \Phi(t,x):=\Theta(t)\Psi(x).\\
\end{array}
\right.
\end{equation}
Using the second assumption in \eqref{condrd} on $d$, we observe that $\delta(x)<0$ for all $x\in [0,1]$. Moreover, we have that $\Theta(t)\to +\infty$ as $t$ tends to $0^+$ and $T^-$.  Under the assumptions \eqref{condrd} and the choice of the parameter $\lambda$, the weight functions $\varphi$ and $\Phi$ defined by \eqref{functcarl} satisfy the following inequalities which are needed in the sequel:
	
	\begin{equation}\label{ineqphi}
		\left\{
		\begin{array}{llll}
			\dis \frac{4}{3}\Phi\leq \varphi\leq\Phi\ \ \mbox{on}\ Q,\\
			\dis 2\Phi\leq \varphi\ \ \mbox{on}\ Q.
		\end{array}
		\right.
	\end{equation}	

The following result is the Caccioppoli's inequality associated to systems \eqref{varrho}-\eqref{rho1}. This result will be also useful for the rest of the paper.

\begin{leme}(Caccioppoli's inequality)\cite{maniar2011}\label{cacc}$ $
	
	Let $\omega^\prime$ be a subset of $\omega_1$ such that $\omega^\prime\Subset\omega_1$. Then, there exists a positive constant $C$ such that
	\begin{equation}\label{caccio}
		\int_{0}^{T}\int_{w^\prime}(\rho^2_x+\varrho_x^2)\ e^{2s\varphi}\,\dq \leq C\int_0^T\int_{\omega_1}s^2\Theta^2(\rho^2+\varrho^2)\ e^{2s\varphi}\ \dq,
	\end{equation}	
	where the weight functions $\varphi$ and $\Theta$ are defined by \eqref{functcarl}.
\end{leme}

We state the following carleman type inequality in the degenerate case, proved in \cite{cannarsa2005, cannarsa2008}.

\begin{prop}\label{prp1}$ $
	
Consider the following system with $f_1\in L^2(Q)$ and $z_T\in L^2(\Omega)$,

\begin{equation}\label{defz}
\left\{
\begin{array}{rllll}
\dis -z_t-(k(x)z_x)_x&=&f_1 &\mbox{in}&Q,\\
z(t,0)=z(t,1)&=&0  &\mbox{on}& (0,T),\\
z(T,\cdot) &=&z_T  &\mbox{in}&  \Omega.
\end{array}
\right.
\end{equation}
Then, there exist two positive constants $C$ and $s_0$, such that every solution of \eqref{defz} satisfies, for all $s\geq s_0$, the following inequality:

\begin{eqnarray}\label{ineqcarl}
\dis \int_{Q}\left(s^3\Theta^3\frac{x^2}{k(x)}z^2+s\Theta k(x)z_x^2\right)e^{2s\varphi}\, \dq\leq C\int_{Q}|f_1|^2e^{2s\varphi}\,\dq\nonumber\\
\dis +Csk(1)\int_{0}^{T}\Theta z_x^2(t,1)e^{2s\varphi(t,1)}\,dt,
\end{eqnarray}	
where $\Theta$ and $\varphi$ are given by \eqref{functcarl}.
\end{prop}

The second result is stated in the following proposition.

\begin{prop}\label{prp2}$ $
	
	Consider the following system with $f\in L^2(Q)$ and $z_T\in L^2(\Omega)$,
	
	\begin{equation}\label{defz1}
	\left\{
	\begin{array}{rllll}
	\dis -z_t-(k(x)z_x)_x+a_0z&=&f &\mbox{in}&Q,\\
	z(t,0)=z(t,1)&=&0  &\mbox{on}& (0,T),\\
	z(T,\cdot) &=&z_T  &\mbox{in}&  \Omega.
	\end{array}
	\right.
	\end{equation}
	Then, there exist two positive constants $C$ and $s_1$, such that every solution of \eqref{defz1} satisfies, for all $s\geq s_1$, the following inequality:
	
	\begin{eqnarray}\label{ineqcarl1}
	\dis \int_{Q}\left(s^3\Theta^3\frac{x^2}{k(x)}z^2+s\Theta k(x)z_x^2\right)e^{2s\varphi}\,\dq\leq C\int_{Q}|f|^2e^{2s\varphi}\,\dq\nonumber\\
	\dis +Csk(1)\int_{0}^{T}\Theta z_x^2(t,1)e^{2s\varphi(t,1)}\,dt.
	\end{eqnarray}
	
\end{prop}

\begin{proof}
 To show the inequality \eqref{ineqcarl1}, we apply the last inequality \eqref{ineqcarl} for the function $f_1=f-a_0z$.
Hence, there are two positive constants $C$ and $s_0$, such that for all $s\geq s_0$, the following inequality holds:
\begin{eqnarray*}
\dis \int_{Q}\left(s^3\Theta^3\frac{x^2}{k(x)}z^2+s\Theta k(x)z_x^2\right)e^{2s\varphi}\,\dq\leq C\int_{Q}|f_1|^2e^{2s\varphi}\,\dq\nonumber\\
\dis +Csk(1)\int_{0}^{T}\Theta z_x^2(t,1)e^{2s\varphi(t,1)}\,dt.
\end{eqnarray*}
On the other hand, using Young inequality, we have 
$$
\int_{Q}|f_1|^2e^{2s\varphi}\,\dq\leq 2\left(\int_{Q}|f|^2e^{2s\varphi}\,\dq+\|a_0\|^2_{\infty}\int_{Q}|z|^2e^{2s\varphi}\,\dq\right).
$$
Now, applying Hardy-Poincaré inequality \eqref{hardy} to the function $e^{s\varphi}z$, the fact that $\dis x\longmapsto\frac{x^2}{k(x)}$ is non-decreasing and using the definition of $\varphi$, we obtain

\begin{equation*}
\begin{array}{rll}
\dis \int_{Q}|z|^2e^{2s\varphi}\,\dq &\leq&\dis  \frac{1}{k(1)}\int_Q\frac{k(x)}{x^2}|z|^2e^{2s\varphi}\,\dq\\
&\leq&\dis  \frac{C}{k(1)}\int_Qk(x)\left((e^{s\varphi}z)_x\right)^2\,\dq\\
&\leq&\dis  \frac{C}{k(1)}\left(\int_Qs^2\lambda^2\Theta^2\frac{x^2}{k(x)}e^{2s\varphi}z^2\,\dq+\int_Qk(x)e^{2s\varphi}z_x^2\,\dq\right).
\end{array}
\end{equation*}
Thus,
\begin{equation*}
\begin{array}{rll}
\dis \int_{Q}|f_1|^2e^{2s\varphi}\,\dq&\leq&\dis 2\int_{Q}|f|^2e^{2s\varphi}\,\dq\\
\dis &&\dis +2\|a_0\|^2_{\infty}\frac{C}{k(1)}\left(\int_Qs^2\lambda^2\Theta^2\frac{x^2}{k(x)}e^{2s\varphi}z^2\,\dq+\int_Qk(x)e^{2s\varphi}z_x^2\,\dq\right)
\end{array}
\end{equation*}
Using the fact that there exist a positive constant $M_1$ such that 
\begin{equation}\label{theta}
1\leq M_1\Theta\ \ \mbox{and}\ \ \Theta^2\leq M_1\Theta^3,	
\end{equation}
we obtain 
\begin{equation*}
\begin{array}{rll}
&&\dis \int_{Q}\left(s^3\Theta^3\frac{x^2}{k(x)}z^2+s\Theta k(x)z_x^2\right)e^{2s\varphi}\,\dq\\ 
&&\leq \dis 2C\int_{Q}|f|^2e^{2s\varphi}\,\dq
\dis +Csk(1)\int_{0}^{T}\Theta z_x^2(t,1)e^{2s\varphi(t,1)}\,dt\\
&&\dis+C_1 \int_{Q}\left(s^2\Theta^3\frac{x^2}{k(x)}z^2+\Theta k(x)z_x^2\right)e^{2s\varphi}\,\dq.
\end{array}
\end{equation*}
Taking $s\geq s_1=\max(s_0, 2C_1)$, we obtain \eqref{ineqcarl1}. This completes the proof.	
\end{proof}

 The next result is concerned with a carleman type inequality in non degenerate case.

\begin{prop}\label{propcarl2}\cite{FursikovImanuvilov}$ $
	
	We consider the following system with  $f\in L^2(Q)$, $a_0\in L^\infty(Q)$ and $k>0$ belong to $\mathcal{C}^2([0,1])$:
	\begin{equation}\label{carl4}
	\left\{
	\begin{array}{rllll}
	\dis -z_t-(k(x)z_x)_x+a_0z &=&f &\mbox{in}&Q_b,\\
	z(t,b_1)=z(t,b_2)&=&0  &\mbox{on}& (0,T),\\
	\end{array}
	\right.
	\end{equation}
	where $Q_b:=(0,T)\times (b_1,b_2)$, $(b_1,b_2)\subset [0,1]$.
	Then, there exist two positive constants $C$ and $s_2$, such that every solution of \eqref{carl4} satisfies, for all $s\geq s_2$, the following inequality:
	
	\begin{eqnarray}\label{ineqcarl2}
	\int_{Q}(s^3\eta^3z^2+s\eta z_x^2)e^{2s\Phi}\,\dq \leq C\left(\int_{Q_b}|f|^2e^{2s\Phi}\,\dq
	+\int_{0}^{T}\int_{\omega_1}s^3\eta^3z^2e^{2s\Phi}\, \dq\right),
	\end{eqnarray}
	
	where the function $\eta$ and $\Phi$ are defined by \eqref{functcarl}.
\end{prop}

\begin{remark}\label{Rmkchange}
By a change of variables $t\mapsto T-t$ in systems \eqref{defz1} and \eqref{carl4}, the inequalities \eqref{ineqcarl1} and \eqref{ineqcarl2} remain true.
\end{remark}

\subsection{An intermediate Carleman estimate}

Now, we state and prove an important result of this paper, which is the intermediate Carleman estimate satisfied by the solutions of systems \eqref{varrho}-\eqref{rho1}. This inequality is obtained by using the Carleman estimates \eqref{ineqcarl1} and \eqref{ineqcarl2}, the Hardy-Poincaré inequality \eqref{hardy} and the Caccioppoli's inequality \eqref{caccio}.

\begin{theorem}\label{thm1}$ $
	
	Assume that the hypotheses \eqref{k} on $k$ are satisfied. Then, there exist two positive constants $C$ and $s_4$, such that every solution $\varrho$ and $\rho$ respectively of \eqref{varrho} and \eqref{rho1} satisfy, for all $s\geq s_4$, the following inequality:
	
	\begin{eqnarray}\label{ineqcarlprinc}
	\int_{Q}\left(s^3\Theta^3\frac{x^2}{k(x)}\varrho^2+s\Theta k(x)\varrho_x^2\right)e^{2s\varphi}\,\dq+\int_{Q}\left(s^3\Theta^3\frac{x^2}{k(x)}\rho^2+s\Theta k(x)\rho_x^2\right)e^{2s\varphi}\,\dq\nonumber\\ \leq C\int_{0}^{T}\int_{\omega_1}s^3\Theta^3(\varrho^2+\rho^2)e^{2s\Phi}\,\dq.
	\end{eqnarray}
\end{theorem}

\begin{proof}
Let us choose an arbitrary open subset $\omega^\prime:=(\alpha,\beta)$ such that $\omega^\prime\Subset\omega_1$. Let us introduce the smooth cut-off function $\xi:\R\to\R$ defined as follows:
	\begin{equation}\label{cutoff}
	\left\{
	\begin{array}{llll}
	\dis 0\leq \xi\leq 1, \quad x\in \R,\\
	\dis \xi(x)=1, \quad x\in [0,\alpha],\\
	\dis \xi(x)=0, \quad x\in [\beta,1].
	\end{array}
	\right.
	\end{equation}	
	Let $\varrho$ and $\rho$ be respectively solutions of \eqref{varrho} and \eqref{rho1}. We set $\widetilde{\varrho}=\xi \varrho$ and $\widetilde{\rho}=\xi \rho$. Then, $\widetilde{\varrho}$ and $\widetilde{\rho}$ are respectively solutions to
\begin{equation}\label{varrhot}
\left\{
\begin{array}{rllll}
\dis \widetilde{\varrho}_t-\left(k(x)\widetilde{\varrho}_x\right)_{x}+a_0\widetilde{\varrho} &=&\dis -\frac{1}{\mu}\widetilde{\rho}\chi_{\O}-\left(k(x)\xi_x\ \varrho\right)_{x}-\xi_x\  k(x)\varrho_x& \mbox{in}& Q,\\
\dis \widetilde{\varrho}(t,0)=\widetilde{\varrho}(t,1)&=&0& \mbox{on}& (0,T), \\
\dis \widetilde{\varrho}(0,\cdot)&=&\dis \frac{1}{\sqrt{\gamma}}\zeta(0,\cdot)&\mbox{in}&\Omega,
\end{array}
\right.
\end{equation}
and 
\begin{equation}\label{rho1t}
\left\{
\begin{array}{rllll}
\dis -\widetilde{\rho}_t-\left(k(x)\widetilde{\rho}_x\right)_{x}+a_0\widetilde{\rho} &=&\dis \widetilde{\varrho}\chi_{\O_d}-\left(k(x)\xi_x\ \rho\right)_{x}-\xi_x\  k(x)\rho_x& \mbox{in}& Q,\\
\dis \widetilde{\rho}(t,0)=\widetilde{\rho}(t,1)&=&0& \mbox{on}& (0,T), \\
\dis \widetilde{\rho}(T,\cdot)&=&\widetilde{\rho}_T&\mbox{in}&\Omega.
\end{array}
\right.
\end{equation}
Applying Proposition \ref{prp2} for $\widetilde{\varrho}$ solution to \eqref{varrhot} with $f=\dis -\frac{1}{\mu}\widetilde{\rho}\chi_{\O}-\left(k(x)\xi_x\ \varrho\right)_{x}-\xi_x\  k(x)\varrho_x$, using Young's inequality and the fact that  $\widetilde{\rho}_x(t,1)=0$, we obtain
	\begin{eqnarray}\label{ineqcarl3rho}
		&&\int_{Q}\left(s^3\Theta^3\frac{x^2}{k(x)}\widetilde{\varrho}^2+s\Theta k(x)\widetilde{\varrho}_x^2\right)e^{2s\varphi}\,\dq\nonumber\\
		&\leq& \dis C\int_{Q}\left|\dis -\frac{1}{\mu}\widetilde{\rho}\chi_{\O}-\left(k(x)\xi_x\ \varrho\right)_{x}-\xi_x\  k(x)\varrho_x\right|^2e^{2s\varphi}\,\dq\\
		&\leq&C_2 \dis \int_{Q}\left[\widetilde{\rho}^2+((k(x)\xi_x{\varrho})_x+k(x)\xi_x{\varrho}_x)^2\right] e^{2s\varphi}\,\dq.\nonumber
	\end{eqnarray}
	Moreover, using again Young's inequality and the definition of the function $\xi$, we have
	
	\begin{eqnarray}\label{ine}
	\int_{Q} ((k(x)\xi_x{\varrho})_x+k(x)\xi_x{\varrho}_x)^2 e^{2s\varphi}\,\dq&=& \int_{Q} ((k(x)\xi_x)_x{\varrho}+2k(x)\xi_x{\varrho}_x)^2 e^{2s\varphi}\,\dq\nonumber\\
	&\leq& \int_{Q}\left[2((k(x)\xi_x)_x)^2{\varrho}^2+8(k(x)\xi_x)^2{\varrho}_x^2\right] e^{2s\varphi}\,\dq\nonumber\\
	&\leq& C_3\int_0^T\int_{\omega^\prime}(\varrho^2+\varrho_{x}^2)\ e^{2s\varphi}\,\dq.
	\end{eqnarray}
	On the other hand, we note that $\dis \frac{x^2}{k(x)}$ is non-decreasing. Applying Hardy-Poincar\'e inequality \eqref{hardy} with the function $e^{s\varphi}\widetilde{\rho}$ and using the definition of $\varphi$, we get:
	\begin{eqnarray*}
		\int_{Q}\widetilde{\rho}^2e^{2s\varphi}\, \dq
		&\leq& \frac{1}{k(1)}\int_{Q} \frac{k(x)}{x^2}\widetilde{\rho}^2e^{2s\varphi}\, \dq\\
		&\leq& \frac{\overline{C}}{k(1)}\int_{Q}k(x)|(\widetilde{\rho}\ e^{s\varphi})_x|^2\, \dq\\
		&\leq& C_4\int_{Q}k(x)\widetilde{\rho}_x^2e^{2s\varphi}\, \dq+C_5\int_{Q}s^2\Theta^2\frac{x^2}{k(x)}\widetilde{\rho}^2e^{2s\varphi}\ \dq.
	\end{eqnarray*}
	Using \eqref{theta}, we get
	
	\begin{eqnarray}\label{ine1}
	\int_{Q}\widetilde{\rho}^2e^{2s\varphi}\, \dq\leq  C_6\int_{Q}\Theta k(x)\widetilde{\rho}_x^2e^{2s\varphi}\, \dq+C_7\int_{Q}s^2\Theta^3\frac{x^2}{k(x)}\widetilde{\rho}^2e^{2s\varphi}\ \dq.
	\end{eqnarray}
	Combining \eqref{ineqcarl3rho}, \eqref{ine}, and \eqref{ine1}, we obtain
	\begin{eqnarray}\label{i1}
	\int_{Q}\left(s^3\Theta^3\frac{x^2}{k(x)}\widetilde{\varrho}^2+s\Theta k(x)\widetilde{\varrho}_x^2\right)e^{2s\varphi}\,\dq
	\leq  C_3\int_0^T\int_{\omega^\prime}(\varrho^2+\varrho_{x}^2)e^{2s\varphi}\,\dq \nonumber\\
	+C_6\int_{Q}\Theta k(x)\widetilde{\rho}_x^2e^{2s\varphi}\, \dq+C_7\int_{Q}s^2\Theta^3\frac{x^2}{k(x)}\widetilde{\rho}^2e^{2s\varphi}\ \dq.
	\end{eqnarray}
	Applying the same way with $\widetilde{\rho}$ solution of \eqref{rho1t}, we obtain
	
	\begin{eqnarray}\label{i2}
	\int_{Q}\left(s^3\Theta^3\frac{x^2}{k(x)}\widetilde{\rho}^2+s\Theta k(x)\widetilde{\rho}_x^2\right)e^{2s\varphi}\,\dq
	\leq  C_8\int_0^T\int_{\omega^\prime}(\rho^2+\rho_{x}^2)e^{2s\varphi}\,\dq\nonumber\\
	+C_9\int_{Q}\Theta k(x)\widetilde{\varrho}_x^2e^{2s\varphi}\, \dq+C_{10}\int_{Q}s^2\Theta^3\frac{x^2}{k(x)}\widetilde{\varrho}^2e^{2s\varphi}\ \dq .
	\end{eqnarray}
	Combining \eqref{i1} and \eqref{i2}, we obtain
	
	\begin{eqnarray*}
		&&\dis \int_{Q}\left(s^3\Theta^3\frac{x^2}{k(x)}\widetilde{\varrho}^2+s\Theta k(x)\widetilde{\varrho}_x^2\right)e^{2s\varphi}\,\dq+\int_{Q}\left(s^3\Theta^3\frac{x^2}{k(x)}\widetilde{\rho}^2+s\Theta k(x)\widetilde{\rho}_x^2\right)e^{2s\varphi}\,\dq\nonumber\\
		&&\leq\dis C_{11}\left(\int_{Q}\Theta k(x)(\widetilde{\varrho}_x^2+\widetilde{\rho}_x^2)e^{2s\varphi}\, \dq+\int_{Q}s^2\Theta^3\frac{x^2}{k(x)}(\widetilde{\varrho}^2+\widetilde{\rho}^2)e^{2s\varphi}\ \dq\right)\\
		 &&\dis + C_{12}\int_0^T\int_{\omega^\prime}(\varrho^2+\rho^2+\varrho^2_x+\rho^2_x)e^{2s\varphi}\,\dq\nonumber.
	\end{eqnarray*}
	Taking $s\geq s_3=\max(s_1,2C_{11})$, we obtain
	
	\begin{eqnarray*}
		\dis \int_{Q}\left(s^3\Theta^3\frac{x^2}{k(x)}\widetilde{\varrho}^2+s\Theta k(x)\widetilde{\varrho}_x^2\right)e^{2s\varphi}\,\dq+\int_{Q}\left(s^3\Theta^3\frac{x^2}{k(x)}\widetilde{\rho}^2+s\Theta k(x)\widetilde{\rho}_x^2\right)e^{2s\varphi}\,\dq\nonumber\\ \leq C_{12}\int_0^T\int_{\omega^\prime}(\varrho^2+\rho^2+\varrho^2_x+\rho^2_x)e^{2s\varphi}\,\dq\nonumber.
	\end{eqnarray*}
	Thanks to Caccioppoli's inequality \eqref{caccio}, this latter inequality becomes
	
	\begin{eqnarray}\label{i5}
		\dis \int_{Q}\left(s^3\Theta^3\frac{x^2}{k(x)}\widetilde{\varrho}^2+s\Theta k(x)\widetilde{\varrho}_x^2\right)e^{2s\varphi}\,\dq+\int_{Q}\left(s^3\Theta^3\frac{x^2}{k(x)}\widetilde{\rho}^2+s\Theta k(x)\widetilde{\rho}_x^2\right)e^{2s\varphi}\,\dq\\ \leq C_{13}\int_0^T\int_{\omega_1}s^2\Theta^2(\varrho^2+\rho^2)e^{2s\varphi}\,\dq\nonumber.
	\end{eqnarray}

	Now let $\overline{\varrho}= \vartheta\varrho$ and $\overline{\rho}=\vartheta\rho$ with $\vartheta=1-\xi$. Then, the support of $\overline{\varrho}$ and $\overline{\rho}$ is contained in $[0,T]\times[\alpha,1]$ and are respectively solutions to

\begin{equation}\label{varrhob}
\left\{
\begin{array}{rllll}
\dis \overline{\varrho}_t-\left(k(x)\overline{\varrho}_x\right)_{x}+a_0\overline{\varrho} &=&\dis -\frac{1}{\mu
	}\overline{\rho}\chi_{\O}-\left(k(x)\vartheta_x\ \varrho\right)_{x}-\vartheta_x\  k(x)\varrho_x& \mbox{in}& Q_\alpha,\\
\dis \overline{\varrho}(t,0)=\overline{\varrho}(t,1)&=&0& \mbox{on}& (0,T), \\
\dis \overline{\varrho}(0,\cdot)&=&\dis \frac{1}{\sqrt{\gamma}}\zeta(0,\cdot)&\mbox{in}&\Omega
\end{array}
\right.
\end{equation}
and 
\begin{equation}\label{rho1b}
\left\{
\begin{array}{rllll}
\dis -\overline{\rho}_t-\left(k(x)\overline{\rho}_x\right)_{x}+a_0\overline{\rho} &=&\dis \overline{\varrho}\chi_{\O_d}-\left(k(x)\vartheta_x\ \rho\right)_{x}-\vartheta_x\  k(x)\rho_x& \mbox{in}& Q_\alpha,\\
\dis \overline{\rho}(t,0)=\overline{\rho}(t,1)&=&0& \mbox{on}& (0,T), \\
\dis \overline{\rho}(T,\cdot)&=&\overline{\rho}_T&\mbox{in}&\Omega,
\end{array}
\right.
\end{equation}
where, $Q_{\alpha}=(0,T)\times(\alpha,1)$. Since on $Q_{\alpha}$ all the above systems are non degenerate, applying Proposition \ref{propcarl2} on $\overline{\varrho}$ solution of \eqref{varrhob} with $b_1=\alpha$, $b_2=1$ and $\dis f=\dis -\frac{1}{\mu}\overline{\rho}\chi_{\O}-\left(k(x)\vartheta_x\ \varrho\right)_{x}-\vartheta_x\  k(x)\varrho_x$, we get
	
	\begin{eqnarray*}
		\int_{Q}\left(s^3\eta^3\overline{\varrho}^2+s\eta \overline{\varrho}_x^2\right)e^{2s\Phi}\,\dq
		\leq C_{14}\int_{Q}\left|\dis -\frac{1}{\mu}\overline{\rho}\chi_{\O}-\left(k(x)\vartheta_x\ \varrho\right)_{x}-\vartheta_x\  k(x)\varrho_x \right|^2e^{2s\Phi}\,\dq\\
		+C_{14}\int_0^T\int_{\omega_1}s^3\eta^3{\varrho}^2e^{2s\Phi}\,\dq.
	\end{eqnarray*}
	Using Young's inequality, we obtain
	
	\begin{eqnarray}\label{i6}
	\int_{Q}\left(s^3\eta^3\overline{\varrho}^2+s\eta \overline{\varrho}_x^2\right)e^{2s\Phi}\,\dq
	\leq C_{15}\int_{Q}[\overline{\rho}^2 +((k(x)\vartheta_x{\varrho})_x+k(x)\vartheta_x{\varrho}_x)^2 ]e^{2s\Phi}\,\dq\nonumber\\
	+C_{14}\int_0^T\int_{\omega_1}s^3\eta^3{\varrho}^2e^{2s\Phi}\,\dq.
	\end{eqnarray}
	Moreover using again Young's inequality and the definition of the function $\vartheta$, we have
	
	\begin{eqnarray}\label{i7}
	\int_{Q} ((k(x)\vartheta_x{\varrho})_x+k(x)\vartheta_x{\varrho}_x)^2 e^{2s\Phi}\,\dq&=& \int_{Q} ((k(x)\vartheta_x)_x{\varrho}+2k(x)\vartheta_x{\varrho}_x)^2 e^{2s\Phi}\,\dq\nonumber\\
	&\leq& \int_{Q}\left[2((k(x)\vartheta_x)_x)^2{\varrho}^2+8(k(x)\vartheta_x)^2{\varrho}_x^2\right] e^{2s\Phi}\,\dq\nonumber\\
	&\leq& C_{16}\int_0^T\int_{\omega^\prime}(\varrho^2+\varrho_{x}^2)e^{2s\Phi}\,\dq,
	\end{eqnarray}
	On the other hand, since $\dis \frac{x^2}{k(x)}$ is non-decreasing, and thanks to Hardy-Poincar\'e inequality \eqref{hardy}, we get
	\begin{eqnarray*}
		\int_{Q}\overline{\rho}^2e^{2s\Phi}\, \dq
		&\leq& \frac{1}{k(1)}\int_{Q} \frac{k(x)}{x^2}(\overline{\rho}e^{s\Phi})^2\, \dq\\
		&\leq& \frac{\overline{C}}{k(1)}\int_{Q}k(x)|(\overline{\rho}e^{s\Phi})_x|^2\, \dq\\
		&\leq& C_{17}\int_{Q}k(x)\overline{\rho}_x^2e^{2s\Phi}\, \dq+C_{18}\int_{Qk(x)}s^2\eta^2\overline{\rho}^2e^{2s\Phi}\ \dq.
	\end{eqnarray*}
	Using \eqref{theta}, the fact that $k\in \mathcal{C}([0;1])$ and $\eta^{-1}\in L^\infty(Q)$, we get

	\begin{eqnarray}\label{i8}
	\int_{Q}\overline{\rho}^2e^{2s\Phi}\, \dq\leq  C_{19}\int_{Q}\eta \overline{\rho}_x^2e^{2s\Phi}\, \dq+C_{20}\int_{Q}s^2\eta^3\overline{\rho}^2e^{2s\Phi}\ \dq.
	\end{eqnarray}
	Combining \eqref{i6}, \eqref{i7} and \eqref{i8}, we obtain
	
	\begin{eqnarray*}
	\int_{Q}\left(s^3\eta^3\overline{\varrho}^2+s\eta \overline{\varrho}_x^2\right)e^{2s\Phi}\,\dq
	\leq  C_{21}\int_0^T\int_{\omega^\prime}(\varrho^2+\varrho_{x}^2)e^{2s\Phi}\,\dq+C_{22}\int_0^T\int_{\omega_1}s^3\Theta^3\varrho^2 e^{2s\Phi}\,\dq\nonumber\\
	+C_{23}\int_{Q}\eta \overline{\rho}_x^2e^{2s\Phi}\, \dq+C_{24}\int_{Q}s^2\eta^3e^{2s\Phi}\ \dq .
	\end{eqnarray*}
	Applying the Caccioppoli's inequality \eqref{caccio} to the latter inequality, we are lead to
	
	\begin{eqnarray}\label{i10}
	\int_{Q}\left(s^3\eta^3\overline{\varrho}^2+s\eta \overline{\varrho}_x^2\right)e^{2s\Phi}\,\dq
	\leq  C_{25}\int_0^T\int_{\omega_1}s^3\Theta^3(\varrho^2+\rho^2)e^{2s\Phi}\,\dq\nonumber\\
	+C_{23}\int_{Q}\eta\overline{\rho}_x^2e^{2s\Phi}\, \dq+C_{24}\int_{Q}s^2\eta^3\overline{\rho}^2e^{2s\Phi}\ \dq .
\end{eqnarray}
Applying the same way to $\overline{\rho}$ solution of \eqref{rho1b} with source term $f=\dis \overline{\varrho}\chi_{\O_d}-\left(k(x)\vartheta_x\ \rho\right)_{x}-\vartheta_x\  k(x)\rho_x$, we obtain
	
	\begin{eqnarray}\label{i11}
	\int_{Q}\left(s^3\eta^3\overline{\rho}^2+s\eta \overline{\rho}_x^2\right)e^{2s\Phi}\,\dq
	\leq  C_{26}\int_0^T\int_{\omega_1}s^3\Theta^3(\varrho^2+\rho^2)e^{2s\Phi}\,\dq\nonumber\\
	+C_{27}\int_{Q}\eta\overline{\varrho}_x^2e^{2s\Phi}\, \dq+C_{28}\int_{Q}s^2\eta^3\overline{\varrho}^2e^{2s\Phi}\ \dq .
	\end{eqnarray}
	Combining \eqref{i10} and \eqref{i11}, we obtain
	
	\begin{eqnarray}\label{i14}
	&&\dis \int_{Q}\left(s^3\eta^3\overline{\varrho}^2+s\eta \overline{\varrho}_x^2\right)e^{2s\Phi}\,\dq+\int_{Q}\left(s^3\eta^3\overline{\rho}^2+s\eta \overline{\rho}_x^2\right)e^{2s\Phi}\,\dq\nonumber\\
	&&\leq \dis C_{29}\left(\int_{Q}\eta(\overline{\varrho}_x^2+\overline{\rho}_x^2)e^{2s\Phi}\, \dq+\int_{Q}s^2\eta^3(\overline{\varrho}^2+\overline{\rho}^2)e^{2s\Phi}\ \dq\right) \\
	&&\dis +C_{30}\int_0^T\int_{\omega_1}s^3\Theta^3(\varrho^2+\rho^2)e^{2s\Phi}\,\dq\nonumber.
	\end{eqnarray}
Taking $s\geq s_4=\max(s_2,2C_{29})$ and using the fact that $\dis e^{3r\sigma(x)}\geq 1$ and $\dis e^{r\sigma(x)}\geq 1$, we obtain	
	
\begin{eqnarray}\label{i16}
&&\dis \int_{Q}\left(s^3\Theta^3\overline{\varrho}^2+s\Theta \overline{\varrho}_x^2\right)e^{2s\Phi}\,\dq+\int_{Q}\left(s^3\Theta^3\overline{\rho}^2+s\Theta \overline{\rho}_x^2\right)e^{2s\Phi}\,\dq\\
&&\dis \leq C_{30}\int_0^T\int_{\omega_1}s^3\Theta^3(\varrho^2+\rho^2)e^{2s\Phi}\,\dq\nonumber.
\end{eqnarray}	
Thanks to \eqref{ineqphi}, the fact that $k\in \mathcal{C}([0,1])$ and the function $\dis \frac{x^2}{k(x)}$ is non-decreasing, one can prove the existence of a constant $C>0$  such that for all $(t,x)\in (0,T)\times [\alpha,1]$, we have
\begin{equation}\label{est}
\begin{array}{llll}
\dis e^{2s\varphi}\leq e^{2s\Phi},\ \ \ \, \frac{x^2}{k(x)}e^{2s\varphi}\leq Ce^{2s\Phi},\ \ \ \, k(x)e^{2s\varphi}\leq Ce^{2s\Phi}.
\end{array}
\end{equation}
Using \eqref{i16} and \eqref{est}, it follows that
	
\begin{eqnarray}\label{i17}
&&\dis \int_{Q}\left(s^3\Theta^3\frac{x^2}{k(x)}\overline{\varrho}^2+s\Theta k(x) \overline{\varrho}_x^2\right)e^{2s\varphi}\,\dq+\int_{Q}\left(s^3\Theta^3\frac{x^2}{k(x)}\overline{\rho}^2+s\Theta k(x) \overline{\rho}_x^2\right)e^{2s\varphi}\,\dq\nonumber\\
&&\dis \leq C_{31}\int_0^T\int_{\omega_1}s^3\Theta^3(\varrho^2+\rho^2)e^{2s\Phi}\,\dq.
\end{eqnarray}		
Combining \eqref{i5} and \eqref{i17}, and using the fact that $e^{2s\varphi}\leq e^{2s\Phi}$, we obtain	

\begin{eqnarray}\label{i18}
&&\dis \int_{Q}\left(s^3\Theta^3\frac{x^2}{k(x)}(\widetilde{\varrho}^2+\overline{\varrho}^2)+s\Theta k(x) (\widetilde{\varrho}_x^2+\overline{\varrho}_x^2)\right)e^{2s\varphi}\,\dq\nonumber\\
&&\dis +\int_{Q}\left(s^3\Theta^3\frac{x^2}{k(x)}(\widetilde{\rho}^2+\overline{\rho}^2)+s\Theta k(x) (\widetilde{\rho}_x^2+\overline{\rho}_x^2)\right)e^{2s\varphi}\,\dq\\
&&\dis \leq C_{32}\int_0^T\int_{\omega_1}s^3\Theta^3(\varrho^2+\rho^2)e^{2s\Phi}\,\dq\nonumber.
\end{eqnarray}		
Using the fact that $\varrho=\widetilde{\varrho}+\overline{\varrho}$ and $\rho=\widetilde{\rho}+\overline{\rho}$, then we have 
\begin{equation}\label{pat}
	|\varrho|^2\leq 2\left(|\widetilde{\varrho}|^2+|\overline{\varrho}|^2\right),\ |\rho|^2\leq 2\left(|\widetilde{\rho}|^2+|\overline{\rho}|^2\right),\ |\varrho_x|^2\leq 2\left(|\widetilde{\varrho}_x|^2+|\overline{\varrho}_x|^2\right),\ |\rho_x|^2\leq 2\left(|\widetilde{\rho}_x|^2+|\overline{\rho}_x|^2\right). 
\end{equation}
Combining \eqref{pat} and \eqref{i18}, we obtain
	
\begin{eqnarray*}
&&\dis \int_{Q}\left(s^3\Theta^3\frac{x^2}{k(x)}\varrho^2+s\Theta k(x) \varrho^2_x\right)e^{2s\varphi}\,\dq
+\int_{Q}\left(s^3\Theta^3\frac{x^2}{k(x)}\rho^2+s\Theta k(x) \rho^2_x\right)e^{2s\varphi}\,\dq\nonumber\\
&&\dis \leq C_{33}\int_0^T\int_{\omega_1}s^3\Theta^3(\varrho^2+\rho^2)e^{2s\Phi}\,\dq.
\end{eqnarray*}
This complete the proof.
\end{proof}

\subsection{An observability inequality result}

This part is devoted to the observability inequality of systems \eqref{varrho}-\eqref{rho1}. This inequality is obtained by using the intermediate Carleman estimate \eqref{ineqcarlprinc}.

\begin{prop} \label{prop5} $ $
	
	 Under the assumptions of Theorem \ref{thm1}, there exist two positive constants $C$ and $s_4$, such that every solution $\varrho$ and $\rho$ of \eqref{varrho} and \eqref{rho1}, respectively, satisfy, for all $s\geq s_4$, the following inequality:
	\begin{eqnarray}\label{obser}
	\int_{Q}\left(s^3\Theta^3\frac{x^2}{k(x)}\varrho^2+s\Theta k(x)\varrho_x^2\right)e^{2s\varphi}\,\dq+\int_{Q}\left(s^3\Theta^3\frac{x^2}{k(x)}\rho^2+s\Theta k(x)\rho_x^2\right)e^{2s\varphi}\,\dq\nonumber\\
	\leq Cs^7\int_{\omega_T}|\rho|^2\,\dq.
	\end{eqnarray}
\end{prop}

\begin{proof}
To get the inequality \eqref{obser}, we will eliminate the local term corresponding to $\varrho$ on the right hand side of \eqref{ineqcarlprinc}.
	\noindent So, let $\omega_2$ be a nonempty open set such that $\omega_1\Subset \omega_2\Subset \O_d\cap\omega$. Let's introduce as in \cite{teresa2000insensitizing}, the cut off function $\xi\in C^{\infty}_0(\Omega)$ such that \begin{subequations}\label{owogene1}
		\begin{alignat}{11}
		\dis 	0\leq \xi\leq 1\ \mbox{in}\ \Omega, \,\,\xi=1 \hbox{ in } \omega_1 ,\,\,  \xi=0 \hbox{ in } \Omega\setminus\omega_2,\label{owo21}\\
		\dis 	\frac{\xi_{xx}}{\xi^{1/2}}\in L^{\infty}(\omega_2),\,\, \frac{\xi_x}{\xi^{1/2}}\in L^{\infty}(\omega_2).\label{owo2}
		\end{alignat}
	\end{subequations}
	Set $\dis u=s^3\Theta^3e^{2s\Phi}$. Then $u(T)=u(0)=0$ and we have the following estimates:
	\begin{equation}\label{con}
		\begin{array}{rll}
			\dis |u\xi|\leq s^3\Theta^3e^{2s\Phi}\xi,\ \ \ \ \ \
			\dis \left|(u\xi)_t\right|\leq Cs^4\Theta^8e^{2s\Phi}\xi,\\
			\\
			\dis |(u\xi)_x|\leq Cs^4\Theta^4e^{2s\Phi}\xi,\ \ \ \ \ \
			\dis |(a(x)(u\xi)_x)_x|\leq Cs^5\Theta^5e^{2s\Phi}\xi,
		\end{array}
	\end{equation}
	where $C$ is a positive constant.\par 
	If we multiply the first equation of \eqref{rho1} by $u\xi\varrho$ and integrate by parts over $Q$, we obtain
	\begin{eqnarray}\label{car4}
	&&\dis -\frac{1}{\mu}\int_{Q}u\xi|\rho|^2\chi_{\O} \ \dT+\int_{Q}\rho\varrho\frac{\partial (u\xi)}{\partial t}\ \dT-\int_{Q}(k(x)(u\xi)_x)_x\ \rho\varrho\ \dT\nonumber\\
	&&-2\int_Qk(x)(u\xi)_x\rho\varrho_x\ \dT=\int_{Q}u\xi|\varrho|^2\chi_{\O_d}\dT.
	\end{eqnarray}
	If we set
	$$\begin{array}{lllll}
	J_1=\dis-\frac{1}{\mu}\int_{Q}u\xi|\rho|^2\chi_{\O} \ \dT,\,
	J_2=\int_{Q}\rho\varrho\frac{\partial (u\xi)}{\partial t}\ \dT,\\
	J_3= \dis -\int_{Q}(k(x)(u\xi)_x)_x\ \rho\varrho\ \dT,\
	J_4=\dis-2\int_Qk(x)(u\xi)_x\rho\varrho_x\ \dT,
	\end{array}
	$$
	then (\ref{car4}) can be rewritten as
	\begin{equation}\label{beau}
	J_1+J_2+J_3+J_4=\int_{Q}u|\xi\varrho|^2\chi_{\O_d}\dT.
	\end{equation}
	Let us estimate $J_i,\ i=1,\cdots,4$. Using the Young's inequality, we have
	
	\begin{eqnarray*}
		J_1&\leq& \dis \frac{1}{\mu}\int_{Q} s^3\Theta^3e^{2s\Phi}\xi|\rho|^2\dT\\
		&\leq&\dis \frac{1}{2\mu^2}\int_Q s^3\Theta^3\frac{x^2}{k(x)}e^{2s\varphi}|\rho|^2\dT+C_{34}\int_{0}^{T}\int_{\omega_2} s^3\Theta^3\frac{k(x)}{x^2}e^{2s(2\Phi-\varphi)}|\rho|^2\dT,	
	\end{eqnarray*}
	
	\begin{eqnarray*}
		J_2&\leq&\dis C_{35}\int_{Q}s^4\Theta^5e^{2s\Phi}\xi|\rho\varrho|\ \dT\\
		&\leq&\dis \frac{\delta_1}{2}\int_Q s^3\Theta^3\frac{x^2}{k(x)}e^{2s\varphi}|\varrho|^2\dT+C_{36}\int_{0}^{T}\int_{\omega_2} s^5\Theta^7\frac{k(x)}{x^2}e^{2s(2\Phi-\varphi)}|\rho|^2\dT,
	\end{eqnarray*}
	
	\begin{eqnarray*}
		J_3&\leq&C_{37}\int_{Q}s^5\Theta^5e^{2s\Phi}\xi|\rho\varrho|\ \dT\\
		&\leq&\dis \frac{\delta_2}{2}\int_Q s^3\Theta^3\frac{x^2}{k(x)}e^{2s\varphi}|\varrho|^2\dT+C_{38}\int_{0}^{T}\int_{\omega_2} s^7\Theta^7\frac{k(x)}{x^2}e^{2s(2\Phi-\varphi)}|\rho|^2\dT,
	\end{eqnarray*}
	
		\begin{eqnarray*}
			J_4&\leq&C_{39}\int_{Q}s^4\Theta^4k(x)e^{2s\Phi}\xi|\rho\varrho_x|\ \dT\\
			&\leq&\dis \frac{\delta_3}{2}\int_Q s\Theta k(x)e^{2s\varphi}|\varrho_x|^2\dT+C_{40}\int_{0}^{T}\int_{\omega_2} s^7\Theta^7k(x)e^{2s(2\Phi-\varphi)}|\rho|^2\dT.
		\end{eqnarray*}
	
	Finally, choosing the constants $\delta_i$ such that $\dis \delta_1=\delta_2=\frac{1}{2C}$ and $\dis \delta_3=\frac{1}{C}$, where $C$ is the constant obtained in Theorem \ref{thm1}, it follows from \eqref{beau} and the previous inequalities that
	\begin{equation}\label{owo4}
	\begin{array}{rll}
	\dis \int_0^T\int_{\omega_1}s^3\Theta^3e^{2s\Phi}|\varrho|^2\dT	
	&\leq&\dis \frac{1}{2C}\int_Q\left( s^3\Theta^3\frac{x^2}{k(x)}|\varrho|^2+s\Theta k(x)|\varrho_x|^2\right)e^{2s\varphi}\dT\\
	&+&\dis \frac{1}{2\mu^2}\int_0^T\int_{\omega_2} s^3\Theta^3\frac{x^2}{k(x)}e^{2s\varphi}|\rho|^2\dT\\
	&+&\dis C_{38}\int_{0}^{T}\int_{\omega_2} s^7\Theta^7\frac{k(x)}{x^2}e^{2s(2\Phi-\varphi)}|\rho|^2\dT\\
	&&\dis +C_{40}\int_{0}^{T}\int_{\omega_2} s^7\Theta^7k(x)e^{2s(2\Phi-\varphi)}|\rho|^2\dT.
	\end{array}
	\end{equation}
	Combining \eqref{ineqcarlprinc} with \eqref{owo4} and taking $\mu$ large enough, we obtain 
	\begin{eqnarray}\label{obser1}
	\dis \int_{Q}\left(s^3\Theta^3\frac{x^2}{k(x)}\varrho^2+s\Theta k(x)\varrho_x^2\right)e^{2s\varphi}\,\dq+\int_{Q}\left(s^3\Theta^3\frac{x^2}{k(x)}\rho^2+s\Theta k(x)\rho_x^2\right)e^{2s\varphi}\,\dq\nonumber\\
	\dis \leq C_{41}\left(\int_{0}^{T}\int_{\omega_2} s^7\Theta^7\frac{k(x)}{x^2}e^{2s(2\Phi-\varphi)}|\rho|^2\dT
	+\int_{0}^{T}\int_{\omega_2} s^7\Theta^7k(x)e^{2s(2\Phi-\varphi)}|\rho|^2\dT\right).
	\end{eqnarray}
Note that $\dis \frac{k(x)}{x^2}$ and $k(x)$ are bounded on $\omega_2$. Furthermore, thanks to \eqref{ineqphi}, we have $2\Phi-\varphi\leq 0$ and consequently, $\dis \Theta^7e^{2s(2\Phi-\varphi)}\in L^\infty(Q)$. Then, using \eqref{obser1}, we obtain
	
		\begin{eqnarray}\label{obser2}
		\int_{Q}\left(s^3\Theta^3\frac{x^2}{k(x)}\varrho^2+s\Theta k(x)\varrho_x^2\right)e^{2s\varphi}\,\dq+\int_{Q}\left(s^3\Theta^3\frac{x^2}{k(x)}\rho^2+s\Theta k(x)\rho_x^2\right)e^{2s\varphi}\,\dq\nonumber\\
		\leq C_{42}s^7\int_0^T\int_{\omega_2}|\rho|^2\,\dq.
		\end{eqnarray}
	Using the fact that, $\omega_2\subset\omega$, we deduce \eqref{obser}.
\end{proof}

Now, we are going to establish the observability inequality of Carleman in the sense that the weight functions do not vanish at $t=0$. We define the functions $\widetilde{\varphi}$ and $\widetilde{\Theta}$ as follows:
\begin{equation}\label{phitil}
\widetilde{\varphi}(t,x)=\left\{
\begin{array}{rllll}
\dis\varphi\left(\frac{T}{2},x\right)\ \ \mbox{if}\ \ t\in \left[0,\frac{T}{2}\right],\\
\dis \varphi(t,x)\ \ \mbox{if}\ \ t\in \left[\frac{T}{2},T\right]
\end{array}
\right.
\end{equation}
and
\begin{equation}\label{Thetatil}
\widetilde{\Theta}(t)=\left\{
\begin{array}{rllll}
\dis\Theta\left(\frac{T}{2}\right)\ \ \mbox{if}\ \ t\in \left[0,\frac{T}{2}\right],\\
\dis \Theta(t)\ \ \mbox{if}\ \ t\in \left[\frac{T}{2},T\right],
\end{array}
\right.
\end{equation}
where the functions $\varphi$ and $\Theta$ are defined in \eqref{functcarl}. In view of the definition of $\varphi$ and $\Theta$, the functions $\widetilde{\varphi}(.,x)$ and $\widetilde{\Theta}(\cdot)$ are non positive and of class $\mathcal{C}^1$ on $[0,T[$. From now on, we fix $s=s_4$.\\
We have the following result.

\begin{prop} \label{pro}$ $
	
	  Under the assumptions of Proposition \ref{prop5}, there exist two positive constants $C=C(\|a_0\|_{L^\infty(Q)},T)$ and $s_4$, and two positive weight functions $\kappa$ and $\hat{\eta}$ such that every solution $\phi$ and $\rho$  of \eqref{varrho} and \eqref{rho1}, respectively, satisfy the following inequality:
	
	\begin{eqnarray}\label{obser3}
	\int_{Q}\kappa^2|\phi|^2\,\dT+\int_{Q}\frac{1}{\hat{\eta}^2}|\rho|^2\,\dT
	\leq C\int_{\omega_T}|\rho|^2\,\dT.
	\end{eqnarray}
	
\end{prop}

\begin{proof}$ $
	
	 We proceed in two steps.\\
\textbf{Step 1.} We prove that there exist a constant $C=C(\|a_0\|_{L^\infty(Q)},T)>0$ and a positive weight function $\hat{\eta}$ such that 
\begin{equation}\label{observ}
\begin{array}{llll}
\dis \int_{Q}\frac{1}{\hat{\eta}^2}|\rho|^2dxdt \leq\dis C\int_{\omega_T}|\rho|^2dxdt.
\end{array}
\end{equation}	
	
	Let us introduce a function $\beta\in \mathcal{C}^1 ([0,T])$  such that
	\begin{equation}\label{hypobeta}
	0\leq \beta\leq 1,\ \beta(t)=1 \hbox{ for } t\in [0,T/2],\  \beta(t)=0 \hbox{ for } t\in [3T/4,T],\  |\beta^\prime(t)|\leq C/T.
	\end{equation}
	For any $(t,x)\in Q$, we set
	$$\begin{array}{lll}
	z(t,x)=\beta(t)e^{-r(T-t)}\rho(t,x),
	\end{array}
	$$
	where $r>0$. Then in view of \eqref{rho1}, the function  $z$ is solution of
	
	\begin{equation}\label{eq26z}
	\left\{
	\begin{array}{rllll}
	\dis -z_t-\left(k(x)z_x\right)_{x}+a_0z+rz &=& \dis \beta e^{-r(T-t)}\varrho\chi_{\O_d}-\beta^\prime e^{-r(T-t)} \rho& \mbox{in}& Q,\\
	\dis z(t,0)=z(t,1)&=&0& \mbox{on}& (0,T), \\
	\dis z(T,\cdot)&=&0&\mbox{in}&\Omega.
	\end{array}
	\right.
	\end{equation}
	Using the classical energy estimates for the system \eqref{eq26z}, we get
	$$
	\begin{array}{llll}
	\dis\frac{1}{2}\|z(0,\cdot)\|_{L^2(\Omega)}^2+\left\|\sqrt{k}z_x\right\|^2_{L^2(Q)}
	+\left(r-\|a_0\|_{L^\infty(Q)}-1\right)\|z\|^2_{L^2(Q)}\\
	\dis\leq \frac{1}{2}\int_0^{3T/4}\int_\Omega |\varrho|^2\ \dq
	+\frac{1}{2}\int_{T/2}^{3T/4}\int_\Omega |\rho|^2 \ \dq.
	\end{array}
	$$
	Hence, if we choose in the latter identity $r$ such that $\dis r=\|a_0\|_{L^\infty(Q)}+\frac{3}{2}$ and using the definition of $\beta$ and $z$, we deduce  that
	$$
	\begin{array}{llll}
	\dis \int_{0}^{T/2}\int_{\Omega}|\rho|^2\ dxdt\leq C(\|a_0\|_{L^\infty(Q)},T)
	\left(\int_0^{3T/4}\int_\Omega |\varrho|^2\ \dq
	+\int_{T/2}^{3T/4}\int_\Omega |\rho|^2\ \dq\right).
	\end{array}
	$$
	Now, using the fact that the functions $\widetilde{\varphi}$ and $\widetilde{\Theta}$ defined by \eqref{phitil} and \eqref{Thetatil} respectively have lower and upper bounds for $(t,x)\in  [0,T/2]\times \Omega$, then introducing the corresponding weight functions in the above expression we get:
	\begin{equation}\label{pou}
	\begin{array}{llll}
	\dis \widetilde{\mathcal{K}}_{[0,T/2]}(\rho)\leq C(\|a_0\|_{L^\infty(Q)},T)
	\left(\int_0^{3T/4}\int_\Omega |\varrho|^2\ \dq
	+\int_{T/2}^{3T/4}\int_\Omega |\rho|^2\ \dq\right),
	\end{array}
	\end{equation}
	where 
	\begin{equation}\label{car3}
	\widetilde{\mathcal{K}}_{[a,b]}(z)=\int_a^b\int_{\Omega}\widetilde{\Theta}^3\frac{x^2}{k(x)}e^{2s_4\widetilde{\varphi}}|z|^2\ \dq.
	\end{equation}
	Adding the term $\widetilde{\mathcal{K}}_{[0,T/2]}(\varrho)$ on both sides of inequality \eqref{pou}, we have
	\begin{equation}\label{pou1}
	\begin{array}{llll}
	&&\dis \widetilde{\mathcal{K}}_{[0,T/2]}(\rho)+\widetilde{\mathcal{K}}_{[0,T/2]}(\varrho)\\
	\\
	&&\leq
	\dis C(\|a_0\|_{L^\infty(Q)},T)
	\left(\int_0^{3T/4}\int_\Omega |\varrho|^2\ \dq
	+\int_{T/2}^{3T/4}\int_\Omega |\rho|^2\ \dq\right)+\widetilde{\mathcal{K}}_{[0,T/2]}(\varrho).
	\end{array}
	\end{equation}
	In order to eliminate the term $\widetilde{\mathcal{K}}_{[0,T/2]}(\varrho)$ in the right hand side of \eqref{pou1}, we use the classical energy estimates for the system \eqref{varrho} and thanks to \eqref{phi}-\eqref{zeta}, we obtain :
	$$
	\begin{array}{llll}
	\dis \int_{0}^{T/2}\int_{\Omega}|\varrho|^2\ \dq\leq\frac{1}{\mu^2}C(\|a_0\|_{L^\infty(Q)},T,\gamma)
	\int_0^{T/2}\int_\Omega |\rho|^2 dx \, dt,
	\end{array}
	$$
	where $C(\|a_0\|_{L^\infty(Q)},T,\gamma)$ is independent of $\mu$. The functions $\widetilde{\varphi}$ and $\widetilde{\Theta}$ have lower and upper bounds for $(t,x)\in  [0,T/2]\times \Omega$. Moreover, the function $\dis x\mapsto \frac{x^2}{k(x)}$ is non-decreasing on $(0;1]$. Then, from the previous inequality we obtain 
	\begin{equation}\label{pou2}
	\dis \widetilde{\mathcal{K}}_{[0,T/2]}(\varrho)
	\leq
	\dis \frac{1}{\mu^2}C(\|a_0\|_{L^\infty(Q)},T,\gamma)
	\int_0^{T/2}\int_\Omega \widetilde{\Theta}^3\frac{x^2}{k(x)}e^{2s_4\widetilde{\varphi}}\ |\rho|^2 dx \, dt.
	\end{equation}
	Combining \eqref{pou2} and \eqref{pou1} with $\mu$ large enough, we obtain
	\begin{equation}\label{pou3}
	\begin{array}{llll}
\dis \widetilde{\mathcal{K}}_{[0,T/2]}(\rho)+\widetilde{\mathcal{K}}_{[0,T/2]}(\varrho)
	\leq
	\dis C(\|a_0\|_{L^\infty(Q)},T)
	\left(\int_{T/2}^{3T/4}\int_\Omega( |\rho|^2+|\varrho|^2)\ \dq\right).
	\end{array}
	\end{equation}
	The functions $\varphi$ and $\Theta$ defined in \eqref{functcarl} have the lower and upper bounds for $(t,x)\in  [T/2,3T/4]\times \Omega$. Moreover, the function $\dis \frac{x^2}{k(x)}$ is non-decreasing on $(0;1]$. Using \eqref{obser}, the relation \eqref{pou3} becomes
	\begin{equation}\label{pou4}
	\begin{array}{llll}\dis \widetilde{\mathcal{K}}_{[0,T/2]}(\rho)+\widetilde{\mathcal{K}}_{[0,T/2]}(\varrho)
	\dis &\leq&\dis  C(\|a_0\|_{L^\infty(Q)},T)
	\left(\mathcal{K}_{[T/2,3T/4]}(\rho)+\mathcal{K}_{[T/2,3T/4]}(\varrho)\right)\\
	\\
	&\leq&\dis C(\|a_0\|_{L^\infty(Q)},T)\int_{\omega_T}|\rho|^2dxdt,
	\end{array}
	\end{equation}
	where
	\begin{equation}\label{}
	\mathcal{K}_{[a,b]}(z)=\int_a^b\int_{\Omega}\Theta^3\frac{x^2}{k(x)}e^{2s_4\varphi}|z|^2\ \dq.
	\end{equation}
	
	On the other hand, since $\Theta=\widetilde{\Theta}$ and $\varphi=\tilde{\varphi}$ in $[T/2,T]\times \Omega$, we use again the estimate \eqref{obser} and we obtain
	\begin{equation}\label{pou5}
	\begin{array}{llll}
	\dis \widetilde{\mathcal{K}}_{[T/2,T]}(\rho)+\widetilde{\mathcal{K}}_{[T/2,T]}(\varrho)&=&\dis \mathcal{K}_{[T/2,T]}(\rho)+\mathcal{K}_{[T/2,T]}(\Psi)\\
	\\
	&\leq&\dis C\int_{\omega_T}|\rho|^2dxdt.
	\end{array}
	\end{equation}
	Adding \eqref{pou4} and \eqref{pou5}, we deduce
	\begin{equation}\label{pou6}
	\begin{array}{llll}
	\dis  \widetilde{\mathcal{K}}_{[0,T]}(\rho)+\widetilde{\mathcal{K}}_{[0,T]}(\varrho)\leq\dis C(\|a_0\|_{L^\infty(Q)},T)\int_{\omega_T}|\rho|^2\ \dq.
	\end{array}
	\end{equation}
	Using the definition of $\widetilde{\mathcal{K}}_{[a,b]}$ given by \eqref{car3}, the inequality \eqref{pou6} becomes
	
	\begin{equation}\label{pou7}
	\begin{array}{llll}
	\dis \int_{Q}\widetilde{\Theta}^3\frac{x^2}{k(x)}e^{2s_4\widetilde{\varphi}}|\rho|^2\ \dq+\int_{Q}\widetilde{\Theta}^3\frac{x^2}{k(x)}e^{2s_4\widetilde{\varphi}}|\varrho|^2\ \dq\dis \leq C(\|a_0\|_{L^\infty(Q)},T)\int_{\omega_T}|\rho|^2\ \dq.
	\end{array}
	\end{equation}
	If we set 
	\begin{equation}\label{defeta}
	\dis \frac{1}{\hat{\eta}^2}=\widetilde{\Theta}^3\frac{x^2}{k(x)}e^{2s_4\widetilde{\varphi}},
	\end{equation}
	then, in view of \eqref{pou7} and \eqref{defeta}, we deduce the estimate \eqref{observ}.\\
	\textbf{Step 2.}
	We prove that there exist a constant $C=C(\|a_0\|_{L^\infty(Q)},T)>0$ and a positive weight function $\kappa$ such that 
	\begin{equation}\label{observ1}
	\begin{array}{rlll}
	\dis \int_{Q} \kappa^2|\phi|^2\ \dT\leq C\int_{\omega_T} |\rho|^2\ \dT.
	\end{array}
	\end{equation}
	
	Let us introduce the function
	\begin{eqnarray}\label{berlin}
	\dis \hat{\varphi}(t)=\min_{x\in \Omega}\widetilde{\varphi}(t,x)
	\end{eqnarray}
and define the weight function $\kappa$  by:

\begin{eqnarray}\label{deftkappa}
	\dis \kappa(t)=e^{s_4\hat{\varphi}(t)}\in L^\infty(0,T).
\end{eqnarray}
	Then $\kappa$ is a positive function of class $\mathcal{C}^1$ on $[0,T)$. Furthermore, $\dis \frac{\partial \hat{\varphi}}{\partial t}$is also a positive function on $[0,T)$. 
	Now, multiplying the first equation of \eqref{phi} by $\kappa^2\phi$ and integrating by parts over $\Omega$, we obtain that 
	\begin{equation*}
	\begin{array}{rlll}
	\dis \frac{1}{2}\frac{d}{dt}\int_{\Omega} \kappa^2|\phi|^2\ dx+\int_{\Omega} \kappa^2k(x)|\phi_x|^2\ dx
	\dis =-\int_{\Omega}\kappa^2\, a_0\,|\phi|^2\ dx
	\dis -\frac{1}{\mu}\int_{\mathcal{O}} \kappa^2\rho\phi\ dx+s_4\dis \int_{\Omega} \kappa^2\frac{\partial\hat{\varphi}}{\partial t}|\phi|^2\ dx
	\end{array}
	\end{equation*}
Applying Young's inequality on the second term of the right hand side of the previous equality, and using the fact that$\dis \frac{\partial \hat{\varphi}}{\partial t}$ is a positive function on $[0,T)$, we deduce that
	\begin{equation*}
	\begin{array}{rlll}
	\dis \frac{d}{dt}\int_{\Omega} \kappa^2|\phi|^2\ dx
	\dis \leq\left(2\|a\|_{L^\infty(Q)}+1\right)\int_{\Omega} \kappa^2|\phi|^2\ dx
	\dis +\frac{1}{\mu^2}\int_{\Omega} \kappa^2|\rho|^2\ dx.
	\end{array}
	\end{equation*}
	Using Gronwall's Lemma and the fact that $\phi(x,0)=0$ for $x\in \Omega$, we obtain that
	\begin{equation}\label{berlin2}
	\begin{array}{rlll}
	\dis \int_{\Omega} |\kappa\phi(t,x)|^2\ dx\leq e^{\left(2\|a\|_{L^\infty(Q)}+1\right)T}\frac{1}{\mu^2}\int_Q \kappa^2|\rho|^2\ \dT,\ \forall t\in [0,T],
	\end{array}
	\end{equation}
Using the definition of $\hat{\varphi}$ and $\kappa$ given by \eqref{berlin} and \eqref{deftkappa} respectively, we have
\begin{equation}\label{ten}
	\kappa^2(t)\leq e^{2s_4\widetilde{\varphi}(t,x)},\ \ \forall x\in \Omega.
\end{equation}
Thanks to the fact that $\dis \widetilde{\Theta}^{-1}\in L^\infty(0,T)$ and that the function $\dis \frac{k(x)}{x^2}$ is non-decreasing on $[0,T)$, then using \eqref{ten} we have
	\begin{equation*}
	\begin{array}{rlll}
	\dis \int_{Q} \kappa^2|\rho|^2\ \dT\leq \int_{Q} \widetilde{\Theta}^3\frac{x^2}{k(x)}e^{2s_4\widetilde{\varphi}}|\rho|^2\ \dT,
	\end{array}
	\end{equation*}
	which combining with \eqref{berlin2} and  \eqref{observ} yields
	\begin{equation*}
	\begin{array}{rlll}
	\dis \int_{Q} |\kappa\phi|^2\ \dT\leq C\int_{\omega_T} |\rho|^2 \dT,
	\end{array}
	\end{equation*}
	where $C=C(\|a_0\|_{L^\infty(Q)},T)>0$.
	Adding the latter inequality with \eqref{observ}, we deduce \eqref{obser3}.	
\end{proof}

\section{Resolution of Problem \ref{pb2}: null controllability problem}\label{null}	
In this section, we are concerned with the proof of Theorem \ref{theo2}. Recall that the main objective is to prove the null controllability of $y^\gamma$ at time $T$. In this section, for any $\gamma>0$, we look for a control $h\in L^2(\omega_T)$ such that the solutions of \eqref{II.28}-\eqref{II.30} satisfies \eqref{y(T)}.

	To prove this null controllability problem, we proceed in three steps using a penalization method.\\
	\noindent\textbf{Step 1.} For any $\varepsilon >0$, we define the cost function:
	\begin{equation}\label{defJ}
	J_{\varepsilon}(h)=\dis \frac{1}{2\varepsilon}\int_{\Omega}|y(T,.;h;v^\gamma(h),0)|^2\ dx+
	\frac{1}{2}\int_{\omega_T}|h|^2\ \dq.
	\end{equation}
	Then we consider the optimal control problem: find $h^\gamma_\varepsilon \in L^2(\omega_T)$ such that
	\begin{equation}\label{opt}
	J_{\varepsilon}(h_\varepsilon^\gamma)=\inf_{\atop h\in L^2(\omega_T)}J_{\varepsilon}(h).
	\end{equation}
	Using minimizing sequences, we can prove that  there exists a unique solution $h^\gamma_\varepsilon$ to \eqref{opt}. Using an Euler-Lagrange first order optimality condition that characterizes the solution $h_\varepsilon$, we can prove that 
	\begin{equation}\label{hgeps}
	h_{\varepsilon}^\gamma=\rho_{\varepsilon}^\gamma\ \ \mbox{in}\ \ \omega_T
	\end{equation}	
	with $\rho^\gamma_\varepsilon$ is the solution of the following system
	\begin{equation}\label{rhogammaeps}
	\left\{
	\begin{array}{rllll}
	\dis -\rho_{t,\varepsilon}^\gamma-\left(k(x)\rho^\gamma_{x,\varepsilon}\right)_{x}+a_0\rho^\gamma_\varepsilon &=&(\psi^\gamma_\varepsilon+\phi^\gamma_\varepsilon)\chi_{\O_d}& \mbox{in}& Q,\\
	\dis \rho^\gamma_\varepsilon(t,0)=\rho^\gamma_\varepsilon(t,1)&=&0& \mbox{on}& (0,T),\\
	\dis \rho^\gamma_\varepsilon(T,\cdot)&=&\dis -\frac{1}{\varepsilon}y(T,.;h^\gamma_\varepsilon;v^\gamma(h^\gamma_\varepsilon),0)&\mbox{in}&\Omega,
	\end{array}
	\right.
	\end{equation}
	where $\psi^\gamma_\varepsilon$ and $\psi^\gamma_\varepsilon$ are solutions, respectively of
	\begin{equation}\label{psigammaeps}
	\left\{
	\begin{array}{rllll}
	\dis \psi_{t,\varepsilon}^\gamma-\left(k(x)\psi^\gamma_{x,\varepsilon}\right)_{x}+a_0\psi^\gamma_\varepsilon &=&0& \mbox{in}& Q,\\
	\dis \psi^\gamma_\varepsilon(t,0)=\psi^\gamma_\varepsilon(t,1)&=&0& \mbox{on}& (0,T), \\
	\dis \psi^\gamma_\varepsilon(0,\cdot)&=&\dis \frac{1}{\sqrt{\gamma}}\zeta^\gamma_\varepsilon(0,\cdot)&\mbox{in}&\Omega,
	\end{array}
	\right.
	\end{equation}
	
	\begin{equation} \label{phigammaeps}
	\left\{
	\begin{array}{rllll}
	\dis \phi_{t,\varepsilon}^\gamma-\left(k(x)\phi^\gamma_{x,\varepsilon}\right)_{x}+a_0\phi^\gamma_\varepsilon &=&\dis -\frac{1}{\mu}\rho^\gamma_\varepsilon\chi_{\O}& \mbox{in}& Q,\\
	\dis \phi^\gamma_\varepsilon(t,0)=\phi^\gamma_\varepsilon(t,1)&=&0& \mbox{on}& (0,T), \\
	\dis \phi^\gamma_\varepsilon(0,\cdot)&=&0&\mbox{in}&\Omega
	\end{array}
	\right.
	\end{equation}
	with $\zeta^\gamma_\varepsilon$ which is solution of
	\begin{equation} \label{zetagammaeps}
	\left\{
	\begin{array}{rllll}
	\dis -\zeta_{t,\varepsilon}^\gamma-\left(k(x)\zeta^\gamma_{x,\varepsilon}\right)_{x}+a_0\zeta^\gamma_\varepsilon &=&\dis \frac{1}{\sqrt{\gamma}}\phi^\gamma_\varepsilon& \mbox{in}& Q,\\
	\dis \zeta^\gamma_\varepsilon(t,0)=\zeta^\gamma_\varepsilon(t,1)&=&0& \mbox{on}& (0,T), \\
	\dis \zeta^\gamma_\varepsilon(T,\cdot)&=&0&\mbox{in}&\Omega
	\end{array}
	\right.
	\end{equation}
and $(y^\gamma_\varepsilon,\ S^\gamma_\varepsilon,\ p^\gamma_\varepsilon,\ q^\gamma_\varepsilon)$ is the solution of systems \eqref{II.28}-\eqref{II.30} associated to the control $v^\gamma_\varepsilon$.\\	
\textbf{Step 2.} 
	If we  multiply  the first equation of \eqref{rhogammaeps}, \eqref{psigammaeps}, \eqref{phigammaeps} and \eqref{zetagammaeps}  by $y^\gamma_\varepsilon$, $S^\gamma_\varepsilon$, $p^\gamma_\varepsilon$ and $q^\gamma_\varepsilon$, respectively  and integrate by parts over $Q$, we successively obtain the following equations:
	
	\begin{equation}\label{calcul1}
	\begin{array}{lllll}
	\dis \int_{\omega_T} h_\varepsilon^\gamma\rho^\gamma_\varepsilon\  \dq+\frac{1}{\varepsilon}\|y(T,.;h^\gamma_\varepsilon;v^\gamma_\varepsilon,0)\|^2_{L^2(\Omega)}=\dis \int_{\O_d^T} y_\varepsilon^\gamma (\phi^\gamma_\varepsilon+\psi^\gamma_\varepsilon)\ \dq+\frac{1}{\mu}\int_{\O_T}q^\gamma_\varepsilon\rho^\gamma_\varepsilon,
	\end{array}
	\end{equation}
	\begin{equation}\label{calcul2}
	\begin{array}{lllll}
	\dis \int_{\O_d^T} y_\varepsilon^\gamma \psi^\gamma_\varepsilon\ \dq&=&\dis \frac{1}{\sqrt{\gamma}}\int_{\Omega}S_\varepsilon^\gamma (0,x)\zeta^\gamma_\varepsilon(0,x)\ dx,
	\end{array}
	\end{equation}
	\begin{equation}\label{calcul3}
	\begin{array}{lllll}
	\dis \int_{\O_d^T} \left((y_\varepsilon^\gamma-z_d+\dis \frac{1}{\sqrt{\gamma}}p_\varepsilon^\gamma\right) \phi^\gamma_\varepsilon\ \dq=-\dis \frac{1}{\mu}\int_{\O_T} q_\varepsilon^\gamma \rho^\gamma_\varepsilon\ \dq
	\end{array}
	\end{equation}
	and
	\begin{equation}\label{calcul4}
	\begin{array}{lllll}
	\dis \frac{1}{\sqrt{\gamma}}\int_{\Omega}S_\varepsilon^\gamma(0,x)\zeta^\gamma_\varepsilon(0,x)\ dx- \dis \frac{1}{\sqrt{\gamma}}\int_{\O_d^T} p_\varepsilon^\gamma \phi^\gamma_\varepsilon\ \dq&=0.
	\end{array}
	\end{equation}
	Combining \eqref{calcul1}-\eqref{calcul4} together with  \eqref{hgeps}, we obtain
	$$
	\|h_\varepsilon^\gamma\|^2_{L^2(\omega_T)}+\frac{1}{\varepsilon}\|y(T,.;h^\gamma_\varepsilon;v^\gamma_\varepsilon,0)\|^2_{L^2(\Omega)}
	=\int_{\O_d^T}z_d\phi^\gamma_\varepsilon\ \dq,
	$$
	which, using Cauchy Schwarz inequality and the fact that  $\dis \frac{1}{\kappa}z_d\in L^2(\O_d^T)$ gives
	\begin{equation}\label{resc1}
     \begin{array}{llll}
	\dis \|h_\varepsilon\|^2_{L^2(\omega_T)}+
	\frac{1}{\varepsilon}\|y(T,.;h^\gamma_\varepsilon;v^\gamma_\varepsilon,0)\|^2_{L^2(\Omega)}
	\leq
	\dis  \left\|\frac{1}{\kappa} z_d\right\|_{L^2(\O_d^T)}\left\|\kappa\phi_\varepsilon^\gamma\right\|_{L^2(Q)}.
	\end{array}
	\end{equation}
	Now, if we apply the Carleman inequality \eqref{obser3} to $\rho_\varepsilon^\gamma$ and $\phi_\varepsilon^\gamma$ solutions of \eqref{rhogammaeps} and \eqref{psigammaeps}, respectively, then there exists $C=C(\|a_0\|_{L^\infty(Q)},T)>0$ such that	
	\begin{eqnarray}\label{resc2}
	\int_{Q}\kappa^2|\phi_\varepsilon^\gamma|^2\ \dT
	\leq C\int_{\omega_T}|\rho_\varepsilon^\gamma|^2\ \dT.
	\end{eqnarray}
	Using \eqref{resc1}, \eqref{resc2} and \eqref{hgeps}, we obtain that 
	$$
	\dis \|h_\varepsilon^\gamma\|^2_{L^2(\omega_T)}+
	\frac{1}{\varepsilon}\|y(T,.;h^\gamma_\varepsilon;v^\gamma_\varepsilon,0)\|^2_{L^2(\Omega)}
	\leq \left\|\frac{1}{\kappa} z_d\right\|_{L^2(\O_d^T)}\|h_\varepsilon\|_{L^2(\omega_T)}.
	$$
	Hence, it follows that,
	\begin{equation}\label{mon10}
	\|h_\varepsilon^\gamma\|_{L^2(\omega_T)}\leq  C  \left\|\frac{1}{\kappa} z_d\right\|_{L^2(\O_d^T)}
	\end{equation}
	and
	\begin{equation}\label{mon10qt}
	\|y(T,.;h^\gamma_\varepsilon;v^\gamma_\varepsilon,0)\|^2_{L^2(\Omega)}\leq  C\sqrt{\varepsilon}\left\|\frac{1}{\kappa} z_d\right\|_{L^2(\O_d^T)},
	\end{equation}
	where $C=C(\|a_0\|_{L^\infty(Q)},T)>0$. \par 
	
	Using the fact that $h_\varepsilon^\gamma$ satisfies \eqref{mon10}, we deduce that $y^\gamma_\varepsilon,\,S^\gamma_\varepsilon, \, p^\gamma_\varepsilon $ and $q^\gamma_\varepsilon$ solutions of \eqref{II.28}-\eqref{II.30} associated to the control $v^\gamma_\varepsilon$ verify the estimates \eqref{kam} of Proposition \ref{pro2}. Then,we can extract subsequences still denoted by $h_\varepsilon^\gamma,\ v_\varepsilon^\gamma,\ y^\gamma_\varepsilon,\,S^\gamma_\varepsilon, \, p^\gamma_\varepsilon $ and $q^\gamma_\varepsilon$ such that when $\varepsilon \rightarrow 0$, we have
	\begin{subequations}\label{convergence2}
		\begin{alignat}{9}
		h_\varepsilon^\gamma&\rightharpoonup& \hat{h}^\gamma&\text{ weakly in }&L^{2}(\omega_T), \label{18k}\\
		v_\varepsilon^\gamma&\rightharpoonup& \hat{v}^\gamma&\text{ weakly in }&L^{2}(\O_T), \label{18}\\
		y_\varepsilon^\gamma&\rightharpoonup& \hat{y}^\gamma&\text{  weakly in }&L^2((0,T);H^1_k(\Omega)), \label{19}\\
		S_\varepsilon^\gamma&\rightharpoonup& \hat{S}^\gamma&\text{  weakly in }&L^2((0,T);H^1_k(\Omega)),\label{20}\\
		q_\varepsilon^\gamma&\rightharpoonup& \hat{q}^\gamma&\text{ weakly  in }&L^2((0,T);H^1_k(\Omega)),\label{21}\\
		p_\varepsilon^\gamma&\rightharpoonup& \hat{p}^\gamma&\text{  weakly in }&L^2((0,T);H^1_k(\Omega))\label{22},\\
		\frac{1}{\sqrt{\gamma} }S\left(0,.;v_\varepsilon^\gamma\right)&\rightharpoonup& \beta &\text{  weakly in }&L^2(\Omega)\label{22a},\\
		y(T,\cdot;h^\gamma_\varepsilon;v^\gamma_\varepsilon,0)&\longrightarrow&0&\hbox{ strongly in }&\ L^2(\Omega).\label{all5}
		\end{alignat}
	\end{subequations}
	
Arguing as in \cite{romario2018, djomegne2018}, using convergences \eqref{convergence2}, we prove that $(\hat{y}^\gamma,\ \hat{S}^\gamma,\ \hat{p}^\gamma\ \hat{q}^\gamma)$ is a solution of \eqref{II.28}-\eqref{II.30} corresponding to the control $\hat{v}^\gamma$ and satisfies \eqref{y(T)}.\\
\textbf{Step 3.} We study the convergence when $\varepsilon \to 0$ of the sequences  $\rho_\varepsilon^{\gamma}$, $\psi_\varepsilon^{\gamma}$, $\phi_\varepsilon^{\gamma}$ and $\zeta_\varepsilon^{\gamma}$.\par
If we apply the Carleman inequality \eqref{obser3} to $\phi_\varepsilon^\gamma$ and $\rho_\varepsilon^\gamma$ solutions of \eqref{rhogammaeps} and \eqref{phigammaeps}, respectively, then there exists a constant $C=C(\|a_0\|_{L^\infty(Q)},T)>0$ such that
 
 \begin{eqnarray}\label{inq}
 \int_{Q}\kappa^2|\phi_\varepsilon^\gamma|^2\,\dT+\int_{Q}\frac{1}{\hat{\eta}^2}|\rho_\varepsilon^\gamma|^2\,\dT
 \leq C\int_{\omega_T}|\rho_\varepsilon^\gamma|^2\,\dT. 
 \end{eqnarray}
 In view of \eqref{hgeps} and \eqref{mon10}, there exists a constant $C=C(\|a_0\|_{L^\infty(Q)},T)>0$ such that
 
 \begin{equation}\label{rhoa}
 \|\rho_\varepsilon^\gamma\|_{L^2(\omega_T)}\leq  C  \left\|\frac{1}{\kappa} z_d\right\|_{L^2(\O_d^T)}.
 \end{equation}
Using \eqref{inq} and \eqref{rhoa}, we obtain 

\label{III.22}
\begin{equation}\label{pas15}
\begin{array}{lll}
\dis \left\|\kappa\phi_\varepsilon^\gamma\right\|_{L^2(Q)}^2+ \left\|\frac{1}{\hat{\eta}}\rho_\varepsilon^\gamma\right\|_{L^2(Q)}^2
\leq
\dis C \left\|\frac{1}{\kappa}\, z_{d}\right\|^2_{L^2(Q)},
\end{array}
\end{equation}
 where $C=C(\|a_0\|_{L^\infty(Q)},T)>0$.\\
  Using the definition of $\widetilde{\varphi}$ and $\widetilde{\Theta}$ given by \eqref{phitil} and \eqref{Thetatil}, respectively, it can be readily seen that there exists a constant $C>0$ such that
  $$
  \kappa\geq C\ \ \mbox{and}\ \ \frac{1}{\hat{\eta}}\geq C
  $$
  and therefore we can obtain
  \begin{equation}\label{inq1}
  \begin{array}{lll}
  \dis \left\|\phi_\varepsilon^\gamma\right\|_{L^2(Q)}^2+ \left\|\rho_\varepsilon^\gamma\right\|_{L^2(Q)}^2
  \leq
  \dis C \left\|\frac{1}{\kappa}\, z_{d}\right\|^2_{L^2(Q)},
  \end{array}
  \end{equation}
  where $C=C(\|a_0\|_{L^\infty(Q)},T)>0$. Using \eqref{psigammaeps}-\eqref{zetagammaeps} and the inequality \eqref{inq1}, we obtain  

\begin{subequations}\label{estim}
	\begin{alignat}{9}
	\|\rho_\varepsilon^\gamma\|_{L^2(Q)}&\leq &C(\|a_0\|_{L^\infty(Q)},T)\left\|\dis\frac{1}{\kappa} z_d\right\|_{L^2(\O_d^T)},\\
	\|\phi_\varepsilon^\gamma\|_{L^2((0,T);H^1_k(\Omega))}&\leq &C(\|a_0\|_{L^\infty(Q)},T,\mu)\left\|\dis\frac{1}{\kappa} z_d\right\|_{L^2(\O_d^T)},\\
	\|\zeta_\varepsilon^\gamma\|_{L^2((0,T);H^1_k(\Omega))}&\leq &C(\|a_0\|_{L^\infty(Q)},T,\mu,\gamma)\left\|\dis\frac{1}{\kappa} z_d\right\|_{L^2(\O_d^T)},\\
	\|\psi_\varepsilon^\gamma\|_{L^2((0,T);H^1_k(\Omega))}&\leq &C(\|a_0\|_{L^\infty(Q)},T,\mu)\left\|\dis\frac{1}{\kappa} z_d\right\|_{L^2(\O_d^T)}.
	\end{alignat}
\end{subequations}
In view of \eqref{estim}, we can extract subsequences still denoted
by $\rho_\varepsilon^\gamma$, $\phi_\varepsilon^\gamma$, $\zeta_\varepsilon^\gamma$ and $\psi_\varepsilon^\gamma$
such that when $\varepsilon \rightarrow 0$, we obtain
\begin{subequations}\label{conver}
	\begin{alignat}{9}
	\rho_\varepsilon^\gamma&\rightharpoonup& \hat{\rho}^\gamma&\text{  weakly in }&L^2(Q),\\
	\psi_\varepsilon^\gamma&\rightharpoonup& \hat{\psi}^\gamma&\text{  weakly in }&L^2((0,T);H^1_k(\Omega)),\\
	\phi_\varepsilon^\gamma&\rightharpoonup& \hat{\phi}^\gamma&\text{ weakly  in }&L^2((0,T);H^1_k(\Omega)),\\
	\zeta_\varepsilon^\gamma&\rightharpoonup& \hat{\zeta}^\gamma&\text{  weakly in }&L^2((0,T);H^1_k(\Omega)).
	\end{alignat}
\end{subequations}

Using \eqref{conver}, we can prove by passing to the limit in systems \eqref{rhogammaeps}-\eqref{zetagammaeps} that the functions $\hat{\rho}^\gamma$, $\hat{\psi}^\gamma$, $\hat{\phi}^\gamma$ and
$\hat{\zeta}^\gamma$ satisfy \eqref{pas45}-\eqref{zetagamma}.
Moreover, using the weak lower semi-continuity of the norm and \eqref{18k}, we deduce from \eqref{mon10} the estimate \eqref{kgamma}.

\section{Conclusion remarks}  \label{conclusion}

In this work, we applied the Stackelberg strategy to control a parabolic equation, with distributed controls that are locally supported in space, under appropriate hypothesis. We considered a linear degenerate heat equation with missing initial condition, and we acted on our system via two controls: a leader and a follower. The Stackelberg method consisted in studying two main problems: a low-regret control problem for the follower, and a null controllability problem for the leader. 
The results obtained here can be extended  to more general degenerate population dynamics models.


\section{Appendix}

\subsection*{Proof of Theorem \ref{exis}}

\begin{proof}
	
	We proceed in three steps.\\	\noindent \textbf{Step 1.} We show the estimate \eqref{esty1y2}.
	Make the change of variable $z(t,x)=e^{-rt}y(t,x),\,\, (t,x)\in Q$,  for some $r>0$ where $y$ is solution to \eqref{eq}. We obtain that $z$ is solution to
	\begin{equation}\label{p1}
		\left\{
		\begin{array}{rllll}
			\dis z_t-(k(x)z_x)_x+a_0z+r z &=&(h\chi_{\omega}+v\chi_{\O})e^{-rt} \qquad &\mbox{in}& Q,\\
			\dis z(t,0)=z(t,1)&=&0& \mbox{on}& (0,T), \\
			\dis z(0,\cdot)&=&g&\mbox{in}&\Omega.
		\end{array}
		\right.
	\end{equation}
	If we multiply the first equation in \eqref{p1} by $z$ and integrate by parts over $Q$, we obtain
	\begin{equation*}
		\begin{array}{rll}
			\dis \int_{Q}z_tz\ dxdt-\int_{Q}(k(x)z_x)_x z\ dxdt+\int_{Q}rz^2\ dxdt=\dis -\int_{Q}a_0z^2\ dxdt
			\dis+ \int_{Q}z(h\chi_{\omega}+v\chi_{\O})e^{-rt}\ dxdt.
		\end{array}
	\end{equation*}
	This latter equality becomes
	\begin{equation}\label{estimation1A1}
		\begin{array}{rlll}
			&&\dis\frac{1}{2}\|z(T,\cdot)\|^2_{L^2(\Omega)}-\frac{1}{2}\|z(0,\cdot)\|^2_{L^2(\Omega)}+\|\sqrt{k(x)}z_x\|^2_{L^2(Q)}+r\| z\|^2_{L^2(Q)}\\
			&&\dis\leq \int_{Q}a_0z^2\ dxdt
			\dis+ \int_{Q}z(h\chi_{\omega}+v\chi_{\O})e^{-rt}\ dxdt.
		\end{array}
	\end{equation}
	We have 
	\begin{equation}\label{ze}
		\int_{Q}a_0z^2\ dxdt\leq \|a_0\|_{\infty}\| z\|^2_{L^2(Q)}.	
	\end{equation}
	Due to the fact that $\dis e^{-rt}\leq 1,\ \forall t\in [0,T]$, we get
	\begin{equation}\label{zea}
		\begin{array}{rlll}
			\dis \int_{Q}z(h\chi_{\omega}+v\chi_{\O})e^{-rt}\ dxdt&\leq&\dis
			\int_{Q}z(h\chi_{\omega}+v\chi_{\O})\ dxdt\\
			&\leq&\dis \| z\|^2_{L^2(Q)}+\frac{1}{2}\|v\|^2_{L^2(\O_T)}+\frac{1}{2}\|h\|^2_{L^2(\omega_T)}.
		\end{array}
	\end{equation}
	Combining \eqref{ze}-\eqref{zea} with \eqref{estimation1A1}, one obtains
	
	\begin{equation*}
		\begin{array}{rlll}
			&&\dis\frac{1}{2}\|z(T,\cdot)\|^2_{L^2(\Omega)}+\frac{1}{2}\|\sqrt{k(x)}z_x\|^2_{L^2(Q)}+\left(r-\|a_0\|_{\infty}-1\right)\| z\|^2_{L^2(Q)}\\
			&&\dis \leq\frac{1}{2}\|g\|^2_{L^2(\Omega)}+\frac{1}{2}\|v\|^2_{L^2(\O_T)}+\frac{1}{2}\|h\|^2_{L^2(\omega_T)}.
		\end{array}
	\end{equation*}
	Taking $r$ such that $\dis r=\|a_0\|_{\infty}+\frac{3}{2}$, we obtain
	\begin{equation*}
		\begin{array}{rlll}
			\dis\|z(T,\cdot)\|^2_{L^2(\Omega)}+\| z\|^2_{L^2((0,T);H^1_k(\Omega))}
			\dis \leq\|g\|^2_{L^2(\Omega)}+\|v\|^2_{L^2(\O_T)}+\|h\|^2_{L^2(\omega_T)}.
		\end{array}
	\end{equation*}
	Since $z=e^{-rt}y$, we deduce the existence  of a constant $C=C(T,\|a_0\|_{\infty})>0$ such that the following estimate holds:
	\begin{equation*}
		\begin{array}{llllll}
			\dis \|y(T,\cdot)\|^2_{L^2(\Omega)}+\|y\|^2_{L^2((0,T); H^1_k(\Omega))}
			\leq 
			C\left(\|v\|^2_{L^2(\O_T)}+\|h\|^2_{L^2(\omega_T)}+\|g\|^2_{L^2(\Omega)}\right)
		\end{array}
	\end{equation*} 	
	and we deduce the inequality \eqref{esty1y2}.
	
	\noindent \textbf{Step 2.} We prove existence by using Theorem \ref{Theolions61}. First of all, it is clear that for any $\phi\in \mathbb{V},$ we have
	$$\|\phi\|_{L^2((0,T);H^1_k(\Omega))}\leq \|\phi\|_{\mathbb{V}}.$$
	This shows that we have the continuous embedding $\mathbb{V}\hookrightarrow L^2((0,T);H^1_k(\Omega))$.
	
	Now, let $\phi \in \mathbb{V}$ and consider the bilinear form $\mathcal{A}(\cdot,\cdot)$ defined on $L^2((0,T);H^1_k(\Omega))\times \mathbb{V}$ by:
	\begin{equation}\label{defCalE}
		\begin{array}{lll}
			\mathcal{A}(y,\phi)&=&\dis -\int_Q y\phi_t \, \dq + \int_Q k(x)y_x\phi_x\,\dq+\int_{Q}a_0 y\phi\, \dq.
		\end{array}
	\end{equation}
	Using Cauchy Schwarz inequality and Remark \ref{rmktrace}, we get that
	$$\begin{array}{lllll}
		\dis |\mathcal{A}(y,\phi)|&\leq&\dis  \|y\|_{L^2(Q)}\|\phi_t\|_{L^2(Q)}+ \|\sqrt{k(x)}y_x\|_{L^2(Q)}\|\sqrt{k(x)}\phi_x\|_{L^2(Q)}+\|a_0\|_{\infty}\|y\|_{L^2(Q)}\|\phi\|_{L^2(Q)}\\
		&\leq& \dis \left[\|\phi_t\|^2_{L^2(Q)}+\|\sqrt{k(x)}\phi_x\|^2_{L^2(Q)}+\|a_0\|^2_{\infty}\|\phi\|^2_{L^2(Q)}\right]^{1/2}\|y\|_{L^2((0,T);H^1_k(\Omega))}.
	\end{array}
	$$
	This means that  there is a constant $C=C(\phi,\|a_0\|_{\infty})>0$  such that
	$$|\mathcal{A}(y,\phi)|\leq C\|y\|_{L^2((0,T);H^1_k(\Omega))}. $$
	Consequently, for every fixed $\phi\in \mathbb{V},$
	the functional  $y\mapsto \mathcal{A}(y,\phi)$ is continuous on $L^2((0,T);H^1_k(\Omega)).$
	
	Next,  we have that for every  $\phi\in\mathbb{V}$,
	\begin{equation}\label{xxl}
		\begin{array}{lllll}
			\mathcal{A}(\phi,\phi)&=&\dis -\int_Q \phi\phi_t \, \dq + \int_Q k(x)\phi_x^2\,\dq+\int_{Q}a_0 \phi^2\, \dq.
		\end{array}
	\end{equation}
	Due to Assumption \ref{aspt}, we get
	$$
	\int_{Q}a_0 \phi^2\, \dq\geq \alpha\|\phi\|^2_{L^2(Q)}.
	$$
	Combining the latter  inequality with \eqref{xxl}, we obtain
	$$
	\begin{array}{lllll}
		\mathcal{A}(\phi,\phi)
		&\geq&\dis \frac{1}{2}\|\phi(0,\cdot)\|^2_{L^2(\Omega)}
		+\dis  \frac{1}{2}\|\sqrt{k(x)}\phi_x\|^2_{L^2(Q)}+\alpha\| \phi\|^2_{L^2(Q)}\\
		&\geq& \dis \min\left\{\frac{1}{2},\alpha\right\}\|\phi\|^2_{\mathbb{V}}.
	\end{array}
	$$
	
	Finally, let us consider the linear functional $\mathcal{L}(\cdot):\mathbb{V}\to \R$ defined by
	$$\begin{array}{lllll}
		\mathcal{L}(\phi):=\dis \int_{Q} (h\chi_{\omega}+v\chi_{\O})\, \phi\; \dq
		+\int_{\Omega} g\,\phi(0,x)\ dx.
	\end{array}
	$$
	Then using Remark \ref{rmktrace}, we obtain
	$$\begin{array}{lll}
		|\mathcal{L}(\phi)|&\leq&\dis \|h\chi_{\omega}+v\chi_{\O}\|_{L^2(Q)}\|\phi\|_{L^2(Q)}+ \|g\|_{L^2(\Omega)}\|\phi(0,\cdot)\|_{L^2(\Omega)}\\
		&\leq&\dis (\|h\chi_{\omega}+v\chi_{\O}\|_{L^2(Q)}+ \|g\|_{L^2(\Omega)})\|\phi\|_{\mathbb{V}}\\
		&\leq & C\|\phi\|_{\mathbb{V}},
	\end{array}
	$$
	where $C= C(T,h,v)>0$.
	Therefore,  $\mathcal{L}(\cdot)$ is continuous on $\mathbb{V}$. Thus,  it follows from Theorem \ref{Theolions61} that there exists $y\in L^2((0,T);H^1_k(\Omega))$ such that
	\begin{equation}\label{formvar}
		\mathcal{A}(y,\phi)= \mathcal{L}(\phi),\quad \forall \phi \in \mathbb{V}.
	\end{equation}
	We have shown that the system \eqref{eq} has a solution $y\in L^2((0,T);H^1_k(\Omega))$ in the sense of Definition \ref{weaksolution}. In addition, using the first equation of \eqref{eq}, we deduce that $y_t\in L^2((0,T);(H^{1}_k(\Omega))^\prime)$. So $y \in W_k(0,T)$ and using \eqref{contWTA}, we have $y\in C([0,T],L^2(\Omega)$. Therefore,  it follows that $y\in \mathcal{H}$.\\
	
	\noindent \textbf{Step 3.} We prove uniqueness.
	Assume that there exist $y_1$ and $y_2$ solutions to \eqref{eq} with the same right hand side $h,\, v$ and initial datum $g$.  Set $z:=e^{-rt}(y_1-y_2)$. Then $z$ satisfies
	\begin{equation}\label{u0}
		\left\{
		\begin{array}{rllll}
			\dis z_t-(k(x)z_x)_x+a_0z+r z &=&0 \qquad &\mbox{in}& Q,\\
			\dis z(t,0)=z(t,1)&=&0& \mbox{on}& (0,T), \\
			\dis z(0,\cdot)&=&0&\mbox{in}&\Omega.
		\end{array}
		\right.
	\end{equation}
	So, if we  multiply the first equation in \eqref{u0} by $z$, and integrate by parts over $Q$, we obtain
	\begin{equation*}
		\begin{array}{rllll}
			\dis \dis\frac{1}{2}\|z(T,\cdot)\|^2_{L^2(\Omega)}+\frac{1}{2}\|\sqrt{k(x)}z_x\|^2_{L^2(Q)}+\left(r-\|a_0\|_{\infty}\right)\| z\|^2_{L^2(Q)}\leq 0.
		\end{array}
	\end{equation*}
	Choosing $\dis r=\|a_0\|_{\infty}+\frac{1}{2}$ in this latter inequality, we deduce that 
	\begin{equation*}
		\begin{array}{rllll}
			\|z\|^2_{L^2((0,T);H^1_k(\Omega))}\leq 0,
		\end{array}
	\end{equation*}
	which means that $z=0$ in $Q$ and consequently, $y_1=y_2$ in $Q$. Therefore, the solution to Problem \eqref{eq} is unique.
	This complete the proof.
\end{proof}

\bibliographystyle{plain}
\bibliography{references}

\end{document}